\documentclass{amsart}  %
\usepackage{longtable}
\usepackage{tabu}
\usepackage{multirow}
\usepackage{booktabs}
\usepackage{mathtools}
\usepackage{amsmath}
\usepackage{graphicx}
\usepackage{color}
\usepackage{amsfonts}
\usepackage{amssymb}
\usepackage{mathrsfs}
\usepackage{amscd}
\usepackage{cite,url}
\usepackage{graphicx}
\usepackage{subcaption}

\usepackage{supertabular}
\usepackage{enumitem}
\setlist[enumerate]{leftmargin=.5in}
\setlist[itemize]{leftmargin=.5in}
\usepackage[toc,page]{appendix}

\newtheorem{theorem}{Theorem}[section]
\newtheorem{prop}[theorem]{Proposition}

\newtheorem{defi}[theorem]{Definition}
\newtheorem{alg}[theorem]{Algorithm}

\newtheorem{thm}[theorem]{Theorem}
\newtheorem{rmk}[theorem]{Remark}
\newtheorem{Example}[theorem]{Example}

\numberwithin{equation}{section}

\numberwithin{equation}{section}

\newcommand{\re}{\mathbb{R}}

\newcommand{\N}{\mathbb{N}}

\newcommand{\lmd}{\lambda}

\newcommand{\nn}{\nonumber}
\newcommand{\eps}{\epsilon}

\newcommand{\dt}{\delta}

\def\gm{\gamma}
\def\rank{\mbox{rank}}

\newcommand{\sig}{\sigma}

\newcommand{\ddd}{,\ldots,}

\newcommand{\reff}[1]{(\ref{#1})}

\newcommand{\mc}[1]{\mathcal{#1}}

\newcommand{\st}{\mathit{s.t.}}

\newcommand{\mt}[1]{\mathtt{#1}}
\newcommand{\mA}{\mathcal{A}}

\newcommand{\mB}{\mathcal{B}}

\newcommand{\trace}{ \mathrm{Trace} }

\newcommand{\bdes}{\begin{description}}
\newcommand{\edes}{\end{description}}

\newcommand{\bal}{\begin{align}}
\newcommand{\eal}{\end{align}}

\newcommand{\bnum}{\begin{enumerate}}
\newcommand{\enum}{\end{enumerate}}

\newcommand{\bit}{\begin{itemize}}
\newcommand{\eit}{\end{itemize}}

\newcommand{\bea}{\begin{eqnarray}}
\newcommand{\eea}{\end{eqnarray}}
\newcommand{\be}{\begin{equation}}
\newcommand{\ee}{\end{equation}}

\newcommand{\baray}{\begin{array}}
\newcommand{\earay}{\end{array}}

\newcommand{\bsry}{\begin{subarray}}
\newcommand{\esry}{\end{subarray}}

\newcommand{\bca}{\begin{cases}}
\newcommand{\eca}{\end{cases}}

\newcommand{\bcen}{\begin{center}}
\newcommand{\ecen}{\end{center}}

\newcommand{\bbm}{\begin{bmatrix}}
\newcommand{\ebm}{\end{bmatrix}}

\newcommand{\bmx}{\begin{matrix}}
\newcommand{\emx}{\end{matrix}}

\newcommand{\bpm}{\begin{pmatrix}}
\newcommand{\epm}{\end{pmatrix}}

\newcommand{\btab}{\begin{tabular}}
\newcommand{\etab}{\end{tabular}}

\begin{document}

\title[Robust Completion for Rank-$1$ Tensors with Noises]
{Robust Completion for Rank-$1$ Tensors with Noises}

\author[Jiawang~Nie]{Jiawang Nie}
\author[Xindong~Tang]{Xindong Tang}
\author[Jinling~Zhou]{Jinling Zhou}

\address{Jiawang~Nie,
Department of Mathematics, University of California San Diego,
9500 Gilman Drive, La Jolla, CA, USA, 92093.}
\email{njw@math.ucsd.edu}

\address{Xindong~Tang, Department of Mathematics,
Hong Kong Baptist University,
Kowloon Tong, Kowloon, Hong Kong.}
\email{xdtang@hkbu.edu.hk}

\address{Jinling~Zhou,
School of Mathematics and Computational Science,
Xiangtan University, Xiangtan, Hunan, 411105, China.}
\email{jinlingzhou@smail.xtu.edu.cn}

\begin{abstract}
This paper studies the rank-$1$ tensor completion problem for cubic tensors
when there are noises for observed tensor entries.
First, we propose a robust biquadratic optimization model
for obtaining rank-$1$ completing tensors.
When the observed tensor is sufficiently close to be rank-$1$,
we show that this biquadratic optimization
produces an accurate rank-$1$ tensor completion.
Second, we give an efficient convex relaxation
for solving the biquadratic optimization.
When the optimizer matrix is separable, we show
how to get optimizers for the biquadratic optimization
and how to compute the rank-$1$ completing tensor.
When that matrix is not separable, we apply its spectral decomposition to
obtain an approximate rank-$1$ completing tensor.
The software {\tt SDPNAL+} is applied to solve the resulting
large size semidefinite programs.
Numerical experiments are given to explore the efficiency of
this biquadratic optimization model and the proposed convex relaxation.
\end{abstract}



\maketitle

\section{Introduction}
Let $\mathbb{R}$ be the real field.
A real tensor $\mathcal{A} \in \re^{n_1\times\cdots\times n_m}$
of order $m$ can be viewed as the multi-array indexed such that
\[
\mathcal{A}=(\mathcal{A}_{i_1\cdots i_m})_{
\substack{
 1 \le i_1 \le n_1, \ldots,
 1 \le i_m \le n_m  .
}  }
\]
When $m=3$, the $\mA$ is called a cubic tensor.
We focus on cubic tensors in this paper.
For vectors $a,b,c$, the notation $a\otimes b\otimes c$
denotes the rank-$1$ tensor such that
\[
( a\otimes b \otimes c  )_{ijk} \, = \, a_i b_j c_k
\]
for all indices $i,j,k$ in the range.
For the tensor $\mathcal{A}$, its {\it Candecomp-Parafac} (CP) rank
\cite{hitchcock1927}
is the smallest integer $r$ such that
\[
\mc{A} = \mA_1 + \cdots + \mA_r,
\]
where each $\mA_i$ is a rank-$1$ tensor.
The CP rank is sometimes just referenced as rank.
There also exist other notions of tensor ranks. We refer to
\cite{carroll1970,LMV2000,Harshman,hitchcock1927,hitchcock1928,Lim13} for related work.

For a partially given tensor, the {\it tensor completion problem} (TCP)
concerns how to fill its missing entries so that the whole tensor
has some attractive properties, e.g., its rank is low.
The whole tensor with all its missing entries filled
is called a tensor completion or completed tensor.
There are wide applications for tensor completions, such as
recommendation systems \cite{frolov2017,Kara-10},
imaging and signal processing \cite{LimCom10,L-H-14,zhao2020}.
More applications can be found in \cite{KolBad09}.

TCPs are natural generalizations of matrix completions \cite{cai2010,candes2012}.
Solving TCPs is generally more challenging.
Indeed, it is shown in \cite{C-D-17} that for rank-$1$ matrix completions,
if the associated bipartite graph is connected,
then the second order Moment-SOS relaxation is tight when the observed matrix is noise-free.
However, there does not exist similar results for rank-$1$ tensor completions,
to the best of the authors's knowledge.
There are different methods for doing tensor completions,
based on different tensor decompositions, such as
CP-decompositions \cite{Bai2016,qiu2021,zhao2020},
Tucker decomposition \cite{ashraphijuo2017,rauhut2017},
tensor singular value decomposition \cite{kilmer2011,ZhangAeron17},
tensor tubal rank \cite{JiangNg19,LAAW20}
and multi-rank \cite{kilmer2011}.
There exist computational methods based on
tensor nuclear norm relaxations
\cite{Nie-nuclear17,tang2015,YuanZhang2016}
and Riemannian-manifold optimization
\cite{Breiding,Dong-Gao-21,gao2024riemannian,Kressner2014,Stein-16,Swijsen22}.
Most tensor computation problems are NP-hard \cite{hillar2013}.

In the earlier work, most tensor completion methods assume the observed tensor
entries are noise-free, i.e., the exact values of those entries are known.
However, in applications, there almost always exist noises.
It is rather desirable to predict tensor data from both noisy and incomplete observations.
In contrast to the abundant literature of noise-free TCPs,
there exist relatively less work for noisy TCPs.
For noisy rank-$1$ matrix completion, stable completions can be computed
by the second order Moment-SOS relaxations \cite{C-D-17}.
Tensor decomposition and total variation regularization based models
have been proposed to solve noisy TCPs \cite{Acar-11,Zhao-Zhang}.
Tensor tubal rank based methods for solving third-order noisy TCPs
are introduced in \cite{wang-jin-17,Wang-19}.
Besides that, a semidefinite relaxation method based on the sum-of-squares hierarchy
is given in \cite{BarMoi22}.

Most tensor optimization problems can be formulated
as polynomial optimization \cite{NiePoly}.
Nuclear norm relaxations can be used to
get tensor completions \cite{Nie-nuclear17,tang2015,zhou}.
Low rank symmetric tensor approximations can be obtained
by solving equations about generating polynomials \cite{NieR}.

\medskip

\noindent{\bf Rank-$1$ tensor completion.}
A particularly interesting class of TCPs is the rank-$1$ tensor completion, i.e.,
finding a rank-$1$ tensor that matches all partially given entries.
The rank-$1$ TCP has very special properties.
It is shown in \cite{zhou} that this problem is equivalent to
a special rank-$1$ matrix recovery problem.
The geometric properties of the rank-$1$ tensor completion problem
are well studied by Kahle et al. \cite{KKKR17}.

For a positive integer $n$, denote the index set $[n] \coloneqq \{1\ddd n\}$.
A cubic tensor $\mc{A}\in \re^{n_1\times n_2\times n_3}$ is said to be {\it partially given}
if there exists a set $\Omega \subseteq [n_1]\times [n_2]\times [n_3]$
such that the entry $\mA_{ijk}$ is given for each triple $(i,j,k) \in \Omega$,
while it is not given for $(i,j,k) \not \in \Omega$.
As shown in \cite{zhou},  $\mA = a \otimes b \otimes c$
is a rank-$1$ tensor completion if and only if the vectors $a,b$ satisfy
\be \label{eq-quadratic-sec1}
\mathcal{A}_{i_sj_sk}a_{i_t}b_{j_t} -
\mathcal{A}_{i_tj_tk}a_{i_s}b_{j_s} \, = \,  0
\ee
for all $(i_s,\,j_s,\,k), (i_t, j_t, k) \in \Omega$.
Based on this observation, a rank-$1$ tensor completion can be obtained by
solving the quadratic system of equations like \reff{eq-quadratic-sec1}.
Nuclear norm and moment relaxation methods are given in \cite{zhou}
to solve \reff{eq-quadratic-sec1}.
Under some general conditions, the moment method can find a rank-$1$ tensor completion
if there exists one, or it detects nonexistence of rank-$1$ completions if it does not exist.

\medskip

\noindent{\bf Contributions.}
In the work \cite{zhou}, the partially given tensor entries are supposed to have no noises.
However, observation of tensor entries often has noises.
In this paper, we consider the rank-$1$ TCP such that the observed tensor entries
$\mA_{ijk}$ for $(i,j,k) \in \Omega$ have noises.
That is, there is a rank-$1$ tensor $\widehat{\mA}\coloneqq
\hat{a}\otimes \hat{b}\otimes \hat{c}$ such that
\be \label{eq:mA=hatmA+dt} \mA = \widehat{\mA} + \dt\mA,\ee
where $\dt\mA$ is the noise tensor.
Our purpose is to recover the rank-$1$ tensor $\hat{a}\otimes \hat{b}\otimes \hat{c}$.
Due to the existence of the noise $\dt \mA$, which is unknown,
the tensor $\hat{a}\otimes \hat{b}\otimes \hat{c}$
cannot be uniquely determined. This is because for any small perturbation
$(\dt \hat{a}, \dt \hat{b}, \dt \hat{c})$, it holds
\[
 \mA =  ( \hat{a} + \dt \hat{a} ) \otimes (\hat{b} + \dt \hat{b} )
 \otimes (\hat{c} + \dt \hat{c} ) +   \dt \widehat{\mA} ,
\]
with $\dt \widehat{\mA}$ also small. For this reason,
it is more practical for us to look for a rank-$1$ tensor
$a^* \otimes b^* \otimes c^*$ such that
\be  \label{eq:tcp}
(\mA - a^*\otimes b^*\otimes c^*)_{ijk} \quad
\mbox{is small for} \quad (i,j,k)\in \Omega.
\ee

Due to the noise, the vectors $a^*, b^*, c^*$
for \reff{eq:tcp} typically do not satisfy equations in \reff{eq-quadratic-sec1}.
This is because $\mA$ itself typically does not have an exact rank-$1$ completion.
Thus, the methods in \cite{zhou} are not applicable
for finding a rank-$1$ tensor completion required for \reff{eq:tcp}.
Alternatively, one may consider to solve the nonlinear least squares (NLS) problem
\be  \label{eq-least-sec1}
\min_{a,b,c} \quad \sum_{(i,j,k)\in\Omega}(\mathcal{A}_{ijk}-a_ib_jc_k)^2.
\ee
This is a nonconvex polynomial optimization problem of degree $6$.
Traditional nonlinear optimization methods are applicable to solve \reff{eq-least-sec1}.
However, their performance is very heuristic,
since they usually can only guarantee to find a stationary point.
In view of polynomial optimization, the objective in \reff{eq-least-sec1}
has degree $6$, with three groups of variables $a,b,c$.
There are totally $n_1+n_2+n_3$ variables in the polynomial.
So, the classical Moment-SOS relaxations \cite{NiePoly}
are quite expensive for solving \reff{eq-least-sec1} directly.

To get a rank-$1$ tensor completion for \reff{eq:tcp},
we propose to solve the following biquadratic optimization problem
\begin{equation}  \label{eq-tcn-1}
\left\{ \baray{cl}
\displaystyle\min  &\sum\limits_{k=1}^{n_3}\sum\limits_{\substack{(i_s, j_s,k)\in\Omega \\ (i_t, j_t,k)\in\Omega}} (\mA_{i_s j_s k} a_{i_t }b_{j_t} -\mA_{i_t j_t k} a_{i_s}b_{ j_s})^2  \\
\st  & ||a||=||b||=1  .
\earay \right.
\end{equation}
The above is motivated as follows. If $a\otimes b\otimes c$
is an approximated rank-$1$ tensor completion for
$\mA = \widehat{\mA} + \dt\mA$, then for each $k$ and for each
$(i_s, j_s,k)\in\Omega$ and $(i_t, j_t,k)\in\Omega$, it holds that
 \[ \mA_{i_s j_s k} a_{i_t }b_{j_t} -\mA_{i_t j_t k} a_{i_s}b_{ j_s} \approx 0. \]
Please see Section~\ref{sc:completion} for more details.
We remark that the problem~(\ref{eq-tcn-1})
only involves the vector variables $a,b$ but not $c$.
Instead of applying the classical Moment-SOS relaxations,
we propose a specifically designed efficient convex relaxation,
which is given in (\ref{moment-for-ab}), to solve (\ref{eq-tcn-1}).
After it is solved, we solve linear least square problems to
retrieve the vector $c$, which eventually produces a rank-$1$ tensor completion for \reff{eq:tcp}.

The main contributions of this paper are:
\begin{itemize}

\item When $\mA$ is sufficiently close to $\hat{a}\otimes \hat{b}\otimes \hat{c}$,
we show that the rank-$1$ tensor $a^* \otimes b^* \otimes c^*$ obtained by solving \reff{eq-tcn-1}
is also close to $\hat{a}\otimes \hat{b}\otimes \hat{c}$.

\item The biquadratic optimization problem (\ref{eq-tcn-1})
is more computationally attractive than (\ref{eq-least-sec1}),
since it only involves variables $(a,b)$ and the degree is four.

\item Due to the biquadratic structure, we propose an efficient convex relaxation,
which is \reff{moment-for-ab}, to solve \reff{eq-tcn-1}.
The relaxation \reff{moment-for-ab} is more computationally appealing
than the general Moment-SOS relaxation for solving \reff{eq-tcn-1}.

\item  We show that the convex relaxation (\ref{moment-for-ab}) is tight
if the matrix $G[y^*]$ in \reff{eq:Gy} is separable,
for the optimizer $y^*$ of \reff{moment-for-ab}.

\item  The software {\tt SDPNAL+} is applied for solving the convex relaxation (\ref{moment-for-ab}),
which is typically a large size semidefinite program (SDP).
Numerical experiments are presented to show the computational efficiency.

\end{itemize}

This paper is organized as follows.
In Section~\ref{sc:preliminaries}, some notation and preliminaries are introduced.
Section~\ref{sc:completion} gives the biquadratic optimization model
and its robustness is shown.
In Section~\ref{sc:minimizer}, we propose a convex relaxation
and show how to get the rank-$1$ completing tensor.
Numerical experiments are presented in Section~\ref{sc:ne}.
We make a conclusion in Section~\ref{sc:con}.

\section{Preliminaries}
\label{sc:preliminaries}

\medskip

\noindent{\bf Notation.}
The symbol $\N$ (resp., $\re$)
denotes the set of nonnegative integers (resp., real numbers).
For an integer $n>0$, $[n]$ denotes the set $\left\{1,\cdots , n\right\}$.
We denote by $e$ the vector of all ones and by $e_i$
the vector with the $i$th entry equaling $1$ and all other equaling zero,
while their dimensions depend on operations with them.
The superscript $^T$ stands for the transpose of a matrix or vector.
For the vector $a\in \re^{n}$, $||a||=\sqrt{a^Ta}$ denotes the Euclidean norm.
For a symmetric matrix $M$, $M\succeq 0$ (resp., $M\succ 0$) means that $M$ is positive semidefinite (resp., positive definite). The cone of all $N$-by-$N$
real symmetric positive semidefinite matrices is denoted as $\mc{S}_+^N$.
For two matrices $A = (a_{ij})$ and $B$, their Kronecker product is the block matrix
\[
A \otimes B =  (a_{ij} B) .
\]
In particular, for two vectors $a =\bbm a_1 & \cdots & a_{n_1} \ebm^T$ and $b$,
\[
a\otimes b=\bbm a_1b^T& a_2b^T & \ldots & a_{n_1}b^T \ebm^T .
\]
For two cubic tensors $\mA,\mB \in \re^{n_1\times n_2\times n_3}$,
their Hilbert-Schmidt product is
\[
\langle\mA,\mB\rangle\,:=\sum_{(i,j,k)\in[n_1] \times [n_2] \times [n_3]}\mA_{ijk}{\mB}_{ijk}.
\]
This induces the Hilbert-Schmidt norm
\[
\Vert \mA\Vert = \sqrt{ \langle\mA,\mA\rangle  }.
\]
For an index set $\Omega\subseteq [n_1] \times [n_2] \times [n_3]$,
$\mA_{\Omega}$ stands for the subtensor of $\mA$
whose entries are indexed by $(i,j,k) \in \Omega$.
Its norm is defined as
\[
\Vert \mA\Vert_{\Omega}  \coloneqq
\sqrt{ \sum_{ (i,j,k) \in \Omega } | \mA_{ijk} |^2 }.
\]

\subsection{Separable matrices}
\label{sc:pre_sep}

We give brief introduction to separable matrices,
which are quite useful for solving the biquadratic optimization \reff{eq-tcn-1}.
Denote the label set
\be \label{df:label:ijkl}
\Gamma:= \{ (i,j,k,l): \,  1\leq i \leq j \leq n_1,\, 1 \leq k \leq l \leq n_2 \}.
\ee
Its cardinality is
$
|\Gamma| = \frac{1}{4} n_1 (n_1+1) n_2 (n_2+1) .
$
Let $\re^{\Gamma}$ be the space of vectors whose entries are labelled by
$(i,j,k,l) \in \Gamma$.
For $i,j\in[n_1]$ and $k,l\in[n_2]$ with $i\ne j$ and $k\ne l$,
denote the basis matrices
\be \label{basis:EijEkl}
\baray{ll}
E_{ii} = e_i e_i^T \in\mc{S}^{n_1}, &
E_{ij} = e_i e_j^T + e_i e_j^T \in \mc{S}^{n_1}, \\
E_{kk} = e_k e_k^T \in\mc{S}^{n_2},   &
E_{kl} = e_k e_l^T + e_l e_k^T \in \mc{S}^{n_2}.
\earay
\ee
Then, for all $a = (a_1, \ldots, a_{n_1})$ and $b = (b_1, \ldots, b_{n_2})$, it holds that
\begin{align}
   (aa^T)\otimes (bb^T) & = \Big( \sum_{i\in [n_1] }\sum_{j \in[n_1]} a_ia_j e_i e_j^T\Big)
    \otimes \Big( \sum_{k\in [n_2] }\sum_{l \in[n_2]} b_kb_l e_k e_l^T \Big) \nn  \\
   & = \sum_{i,j\in[n_1]}\sum_{k,l\in[n_2]} a_ia_jb_kb_l (e_i e_j^T)\otimes (e_k e_l^T) \nn  \\
   & = \sum_{(i,j,k,l)\in\Gamma } a_i a_j b_k b_l E_{ij} \otimes E_{kl}.  \label{eq:ababT}
\end{align}
This motivates us to define the linear matrix function in $y \in \re^{\Gamma}$:
\be  \label{eq:Gy}
G[y]  \coloneqq  \sum_{(i,j,k,l)\in\Gamma }y_{ijkl}  E_{ij} \otimes E_{kl},
\ee
where $y = (y_{ijkl})_{ (i,j,k,l)\in\Gamma }$.
For instance, when $n_1=n_2=3$, $G[y]$ reads as
\[
G[y] = \left[
\begin{array}{ccccccccc}
y_{1111} & y_{1112} & y_{1113} & y_{1211} & y_{1212} & y_{1213} & y_{1311} & y_{1312} & y_{1313}\\
y_{1112} & y_{1122} & y_{1123} & y_{1212} & y_{1222} & y_{1223} & y_{1312} & y_{1322} & y_{1323}\\
y_{1113} & y_{1123} & y_{1133} & y_{1213} & y_{1223} & y_{1233} & y_{1313} & y_{1323} & y_{1333}\\
y_{1211} & y_{1212} & y_{1213} & y_{2211} & y_{2212} & y_{2213} & y_{2311} & y_{2312} & y_{2313}\\
y_{1212} & y_{1222} & y_{1223} & y_{2212} & y_{2222} & y_{2223} & y_{2312} & y_{2322} & y_{2323}\\
y_{1213} & y_{1223} & y_{1233} & y_{2213} & y_{2223} & y_{2233} & y_{2313} & y_{2323} & y_{2333}\\
y_{1311} & y_{1312} & y_{1313} & y_{2311} & y_{2312} & y_{2313} & y_{3311} & y_{3312} & y_{3313}\\
y_{1312} & y_{1322} & y_{1323} & y_{2312} & y_{2322} & y_{2323} & y_{3312} & y_{3322} & y_{3323}\\
y_{1313} & y_{1323} & y_{1333} & y_{2313} & y_{2323} & y_{2333} & y_{3313} & y_{3323} & y_{3333}\\
\end{array} \right].
\]

The matrix $G[y]$ is closely related to biquadratic optimization.
If there is a Borel measure $\mu$ whose support is contained in a set
$K \subseteq \re^{n_1}\times \re^{n_2}$ such that
\be\label{eq:yijlk}
y_{ijkl} = \int_K a_ia_jb_kb_l {\tt d}\mu
\quad \text{for all} \quad  (i,j,k,l)\in \Gamma,
\ee
then \reff{eq:ababT} implies
\be  \label{eq:Gy_moment}
\begin{aligned}
G[y] = \sum_{(i,j,k,l)\in\Gamma } \int_K a_i a_j b_k b_l\, {\tt d}\mu E_{ij} \otimes E_{kl}
  = \int_K  (aa^T)\otimes (bb^T) {\tt d}\mu.
\end{aligned}
\ee

The matrix $G[y]$ is said to be {\it separable} if there exist vectors
$a^{(i)}\in\re^{n_1}$, $b^{(i)}\in\re^{n_2}$ and
scalars $\lmd_i \in \re^1$ such that
\be \label{decomp:G[y]}
\begin{gathered}
G[y] = \sum\limits_{i=1}^r \lmd_i
( a^{(i)} {a^{(i)} }^T )  \otimes  ( b^{(i)} {b^{(i)} }^T ) , \\
\Vert a^{(i)}\Vert = \Vert b^{(i)}\Vert = 1,  \,  \lmd_i > 0, \, i=1, \ldots, r .
\end{gathered}
\ee
It is interesting to note that $G[y]$ is separable if and only if \reff{eq:Gy_moment}
holds for some measure $\mu$ whose support is contained in
$K = \{ (a,b) : \Vert a\Vert = \Vert b\Vert = 1 \}$.
We refer to \cite{NieZhang16} for this fact.
Clearly, if $G[y]$ is separable, then it must be positive semidefinite.

The case that $\rank\, G[y] = 1$ is particularly interesting.

\begin{prop} \label{prop:rank1}
Let $y\in \re^{\Gamma}$ be such that $G[y] \succeq 0$.
If $\rank \, G[y] = 1$ and $\trace\, G[y] = 1$,
then there exist $a\in\re^{n_1}$ and $b\in\re^{n_2}$ such that
\be \label{G[y]=aaTbbT}
 G[y] = (aa^T)\otimes (bb^T),\quad \Vert a\Vert = \Vert b\Vert = 1.
\ee
\end{prop}
\begin{proof}
We write $G[y]$ in the block form
\[
G\left[y\right] \, = \, \begin{bmatrix}
 B_{11} & B_{12} & \cdots  &  B_{1n_1}\\
 B_{12} & B_{22} & \cdots  &  B_{2n_1}\\
 \vdots & \vdots & \ddots  & \vdots\\
 B_{1n_1} & B_{2n_1} & \cdots  &  B_{n_1n_1}\\
\end{bmatrix},
\]
where each block $B_{ij}\in \re^{n_2\times n_2}$ is also a symmetric matrix.
Since $\rank \,G\left[y \right]=1$, we know $\rank \,B_{ij} \leq 1$ for all $i,j$.
Every diagonal block $B_{ii} \succeq 0$, since $G[y] \succeq 0$.
By (\ref{eq:Gy}), it holds that for all $i,j \in [n_1]$
\[ B_{ij}=\sum_{1\le k\le l\le n_2} y_{ijkl} E_{kl}.\]
Note that each $B_{ij}$ is also a symmetric matrix.
Therefore, for all $1\le i\le j\le n_1$, there exist $\mu_{ij}\in\mathbb{R}$
and $w_{ij}\in \mathbb{R}^{n_1}$ such that
\be \label{decomp:Bij}
B_{ij}=\mu_{ij}w_{ij}w_{ij}^T,\quad \Vert w_{ij}\Vert = 1,
\quad e^Tw_{ij} \ge0 , \quad \mu_{ii}\geq 0.
\ee
Since $G[y]\succeq 0$ and $\mbox{Trace}\, G[y] = 1$,
we can generally assume that $B_{11}$ is a nonzero matrix
(if $B_{11}=0$, then $B_{1i}=0$ for all $i$,
in which case we can consider the first nonzero block $B_{ii} \ne 0$).
Then $\mu_{11} > 0$ and $w_{11}$ is a nonzero vector. Note that
\[
\begin{bmatrix}
 B_{11} & B_{1i}    \\
 B_{1i} & B_{ii}
\end{bmatrix}
\succeq 0, \quad
\rank \begin{bmatrix}
 B_{11} & B_{1i}    \\
 B_{1i} & B_{ii}
\end{bmatrix}  = 1.
\]
This forces that each $w_{1i}$ is a multiple of $w_{11}$ or $\mu_{1i}=0$.
Since $\Vert w_{1i}\Vert = 1$ and $e^Tw_{1i}\ge 0$,
we have either $w_{1i} = w_{11}$ or $\mu_{1i} = 0$ for each $i$.
Since $\rank \, G[y] = 1$, each block row must be a multiple of the first row block.
Therefore, we can further get that
either $w_{ij} = w_{11}$ or $\mu_{ij} = 0$, for all $i,j$.
For the case $\mu_{ij} = 0$, we can also set $w_{ij} = w_{11}$
and \reff{decomp:Bij} still holds.

Let $U = (\mu_{ij})_{i,j \in [n_1]}$, then
\[
G[y] \, = \, U \otimes  (w_{11} w_{11}^T ).
\]
Since $G[y] \succeq 0$ and $\rank \, G[y] = 1$,
we also have $U \succeq 0$ and $\rank \, U = 1$.
There exists a vector $a \in \re^{n_1}$ such that $U = aa^T$.
Let $b = w_{11}$.  Since $\trace\, G[y] = 1$, we have
$\| a \| = \| b \| = 1$.
This completes the proof.	
\end{proof}

When $\rank\, \, G[y] > 1$, the separability of $G[y]$
can be detected by Algorithm~4.2 in \cite{NieZhang16}.
It requires to solve a hierarchy of moment relaxations.
If $G[y]$ is separable, the algorithm can get a decomposition like \reff{decomp:G[y]}.
If it is not separable, the algorithm can return a certificate for that,
i.e., it shows a moment relaxation is infeasible
and denies the existence of \reff{decomp:G[y]}.
We refer to \cite{NieZhang16} for more details about separable matrices.

\section{Biquadratic optimization for completion }
\label{sc:completion}

Let $\mc{A} \in \re^{n_1 \times n_2  \times n_3}$
be a partially given tensor such that its entries
$\mA_{ijk}$ are known for $(i,j,k) \in \Omega$,
for a given index set $\Omega\subseteq [n_1]\times [n_2] \times [n_3]$.
We wish to look for a rank-$1$ tensor $a \otimes b \otimes c$
with $a  \in \re^{n_1}$, $b \in \re^{n_2}$, $c \in \re^{n_3}$, such that
$\mA_{ijk} = a_i b_j c_k$ for all $(i,j,k) \in \Omega$.
However, such vectors $a,b,c$ usually do not exist
when there are noises in the observed tensor entries $\mA_{ijk}$.
Typically, $\mA$ is close to a rank-$1$ tensor, but they are not exactly equal.

Suppose there is a rank-$1$ tensor
$\widehat{\mA}=\hat{a}\otimes \hat{b}\otimes \hat{c}$,
for nonzero vectors $\hat{a}\in\re^{n_1}$,
$\hat{b}\in\re^{n_2}$, $\hat{c}\in\re^{n_3}$ such that
\[
\mA_{ijk}=\widehat{\mA}_{ijk}+\delta\mA_{ijk} \quad
\text{for all} \quad  (i,\,j,\,k)\in\Omega,
\]
where $\delta\mA_{ijk}$ stands for the noise at the $(i,j,k)$th tensor entry.
Typically, the noises $\delta\mA_{ijk}$ are small.
Throughout the context, we say the noise $\delta\mA$
is small if its norm $\Vert \delta\mA\Vert_{\Omega}$ is small.
Up to scaling, we can generally assume that
\be  \label{eq:norm=1}
\Vert \hat{a}\Vert = \Vert\hat{b}\Vert = 1.
\ee
This is because $\hat{a}\otimes \hat{b}\otimes \hat{c} =
(\alpha\hat{a})\otimes (\beta\hat{b})\otimes (\frac{1}{\alpha\beta}\hat{c})$
for all nonzero scalars $\alpha,\beta$.
Our goal is to recover the vectors $\hat{a}$, $\hat{b}$, $\hat{c}$ from
the given tensor entries $\mA_{ijk}$.

For each $k=1\ddd n_3$, denote the index set of pairs
\begin{equation*}
\Omega_k \, \coloneqq \, \left\{\  (i,j)\in [n_1]\times [n_2] : (i,j,k)\in \Omega\ \right\}.
\end{equation*}
Let $m_k\coloneqq |\Omega_k|$, the cardinality of $\Omega_k$.
One can write that
\be  \label{Omega_k}
\Omega_k  =  \{ (i_1,j_1),\,(i_2,j_2),\ldots, (i_{m_k},j_{m_k}) \}.
\ee
When all $\delta\mA_{ijk} = 0$,
for each $k = 1 \ddd n_3$ we have
\be  \label{eq-solve-ck}
\begin{bmatrix}
\mathcal{A}_{i_1j_1k}\\
\mathcal{A}_{i_2j_2k}\\
\vdots \\
 \mathcal{A}_{i_{m_k}j_{m_k}k}
\end{bmatrix}
=
\begin{bmatrix}
\hat{a}_{i_1}\hat{b}_{j_1}\\
\hat{a}_{i_2}\hat{b}_{j_2}\\
\vdots \\
\hat{a}_{i_{m_k}}\hat{b}_{j_{m_k}}
\end{bmatrix}\hat{c}_k .
\ee
This implies that
\[
\mA_{i_sj_sk} \hat{a}_{i_t} \hat{b}_{j_t} -
\mA_{i_tj_tk} \hat{a}_{i_s} \hat{b}_{j_s} \, =  \, 0
\]
for all $(i_s,\,j_s), (i_t,\,j_t) \in \Omega_k$.
However, when there are noises, the above equations do not hold anymore.

When there are noises $\delta\mA_{ijk}$, the equation~\eqref{eq-solve-ck} becomes
\begin{equation}\label{eq:noise:mataix}
\begin{bmatrix}
\mathcal{A}_{i_1j_1k} \\
\mathcal{A}_{i_2j_2k} \\
\vdots  \\
\mathcal{A}_{i_{m_k}j_{m_k}k}
\end{bmatrix}
=
\begin{bmatrix}
\hat{a}_{i_1}\hat{b}_{j_1}\\
\hat{a}_{i_2}\hat{b}_{j_2}\\
\vdots \\
\hat{a}_{i_{m_k}}\hat{b}_{j_{m_k}}
\end{bmatrix}\hat{c}_k
+\bbm
\delta\mA_{i_1j_1k}\\
\delta\mA_{i_2j_2k}\\
\vdots\\
\delta\mA_{i_{m_k}j_{m_k}k}
\ebm.
\end{equation}
When the noises $\delta\mA_{ijk}$ are small, one can see that
\[
\begin{array}{rl}
& \mA_{i_sj_sk}\hat{a}_{i_t}\hat{b}_{j_t} -
\mA_{i_tj_tk}\hat{a}_{i_s}\hat{b}_{j_s} \\
= & \hat{c}_k(\hat{a}_{i_s}\hat{b}_{j_s}\hat{a}_{i_t}\hat{b}_{j_t}-
\hat{a}_{i_t}\hat{b}_{j_t}\hat{a}_{i_s}\hat{b}_{j_s})
+(\delta\mA_{i_sj_sk} \hat{a}_{i_t}\hat{b}_{j_t}-
\delta\mA_{i_tj_tk} \hat{a}_{i_s}\hat{b}_{j_s}) \\
= & \delta\mA_{i_sj_sk} \hat{a}_{i_t}\hat{b}_{j_t} -
\delta\mA_{i_tj_tk} \hat{a}_{i_s}\hat{b}_{j_s} \\
\approx & 0 .
\end{array}
\]
This motivates us to find a rank-$1$ tensor completion $a \otimes b \otimes c$ such that
\[
\mA_{i_sj_sk} a_{i_t} b_{j_t} -
\mA_{i_tj_tk} a_{i_s} b_{j_s} \approx 0.
\]
So, we consider the set of quadratic homogeneous polynomials in $a,b$:
\be \label{polyset:Phi}
\Phi  \coloneqq \left\{ \mA_{i_sj_sk}a_{i_t}b_{j_t} -
\mA_{i_tj_tk}a_{i_s}b_{j_s}
\left| \baray{c}
(i_s,\,j_s), (i_t,\,j_t) \in \Omega_k, \\
s <  t, \, k = 1, \ldots,  n_3
\earay \right.
\right\} .
\ee
There are totally $\sum_{k=1}^{n_3} \binom{m_k}{2}$ polynomials in $\Phi$.
Each $\phi\in\Phi$ is bilinear in $(a,b)$.
When $\mA$ is noise-free, i.e., $\delta\mA_{ijk} = 0$ for all $(i,j,k)\in\Omega$, we have
$\phi(\hat{a},\hat{b}) = 0$ for every $\phi\in\Phi$.
When the noises are small, we have
$\phi(\hat{a},\hat{b})  \approx  0$.

To get a rank-$1$ completing tensor $\widehat{\mA}=\hat{a}\otimes \hat{b}\otimes \hat{c}$
as in \eqref{eq:norm=1} and \eqref{eq:noise:mataix}, we consider the
polynomial optimization problem
\be    \label{pri-find-ab}
\left\{ \begin{array}{cl}
 \min & \displaystyle f(a,b)  \coloneqq
\sum\limits_{k=1}^{n_3}\sum\limits_{1 \le s < t \le m_k}
(\mA_{i_s j_s k} a_{i_t }b_{j_t} -\mA_{i_t j_t k} a_{i_s}b_{ j_s})^2  \\
\st & \Vert a\Vert^2 = \Vert b\Vert^2 = 1.
\end{array}
\right.
\ee
The objective polynomial $f$ is biquadratic in $(a,b)$, so it holds
\be\label{eq:fab=f-a-b}
f(a,b) = f(-a,b) = f(a,-b) = f(-a,-b).
\ee
This means that $(a^*,b^*)$ is a minimizer for \reff{pri-find-ab}
if and only if all four pairs $(\pm a^*, \pm b^*)$ are minimizers.
Since \reff{pri-find-ab} is a polynomial optimization problem,
one can apply the classical Moment-SOS hierarchy \cite{HKL20,NiePoly} to solve it.
However, this is not computationally appealing for solving \reff{pri-find-ab}.
In Section~\ref{sc:minimizer}, we propose an efficient convex relaxation,
which is based on the biquadratic structure.

\subsection{Stability analysis}
\label{sec:error}

Suppose $(a^*,b^*)$ is the minimizer of \eqref{pri-find-ab}.
We show that $a^*\otimes b^*$ is close to $\hat{a} \otimes \hat{b}$
or $-\hat{a} \otimes \hat{b}$
when the noise $\dt \mA$ is sufficiently small.

Recall that $\widehat{\mA} = \hat{a}\otimes \hat{b}\otimes \hat{c}$
with the normalization $\Vert \hat{a}\Vert = \Vert \hat{b}\Vert = 1$.
For $k=1,\ldots, n_3$ and for $(i_t,j_t),\,(i_s,j_s)\in \Omega_k$,
let $H_{s,t, k}$ denote the Hessian of
\be   \label{hat:phi}
\hat{\psi}_{s, t, k} (a, b) \, \coloneqq \,
 (\widehat{\mA}_{i_sj_sk}a_{i_t}b_{j_t} -
\widehat{\mA}_{i_tj_tk}a_{i_s}b_{j_s})^2
\ee
at the point $(\hat{a},\hat{b})$, i.e.,
\[
H_{s, t, k}  \, = \,  \nabla^2 \hat{\psi}_{s, t, k} (\hat{a}, \hat{b} ) .
\]
Since $\widehat{\mc{A}}_{ijk} = \hat{a}_i\hat{b}_j\hat{c}_k$, one can verify that
\[
H_{s,t,k}  =  \hat{z}_{s,t, k} \cdot ( \hat{z}_{s,t, k} )^T,
\]
where $\hat{z}_{s,t, k}$ is the column vector
\[
 \hat{z}_{s,t, k}  =  \sqrt{2}  \hat{c}_k
 \left[\begin{array}{cc}
\hat{a}_{i_s}\hat{b}_{j_s}\hat{b}_{j_t} e_{i_t} - \hat{a}_{i_t}\hat{b}_{j_t}\hat{b}_{j_s}e_{i_s}  \\
\hat{a}_{i_s}\hat{b}_{j_s}\hat{a}_{i_t} e_{j_t} - \hat{a}_{i_t}\hat{b}_{j_t}\hat{a}_{i_s}e_{j_s}
\end{array}\right] \in \re^{n_1 + n_2 }.
\]
When $\hat{a}_{i_s}\hat{b}_{j_s}\hat{a}_{i_t}\hat{b}_{j_t} \ne 0$, we have the expression
\[
\hat{z}_{s,t, k} \, =  \, \sqrt{2}   \hat{c}_k
\hat{a}_{i_s}\hat{b}_{j_s}\hat{a}_{i_t}\hat{b}_{j_t}
\bbm
\frac{e_{i_t}}{\hat{a}_{i_t}}  - \frac{e_{i_s}}{\hat{a}_{i_s}}  \\
\frac{e_{j_t}}{\hat{b}_{j_t}}  - \frac{e_{j_s}}{\hat{b}_{j_s}}
\ebm  .
\]
For the function
\be
\hat{\psi} (a, b) \, \coloneqq \,
\sum_{k=1}^{n_3} \sum_{1 \le s < t \le m_k}   \hat{\psi}_{s,t,k} (a, b),
\ee
its Hessian at $(\hat{a}, \hat{b})$ is the matrix
\be\label{eq:Hab}
H(\hat{a},\hat{b}) \, \coloneqq  \,  \sum_{k=1}^{n_3}  \sum_{1 \le s < t \le m_k}
  \hat{z}_{s,t, k}    \cdot ( \hat{z}_{s,t, k}   )^T  .
\ee
Let $\hat{Z}$ be the  matrix whose rows are the vectors $(\hat{z}_{s,t, k})^T$, i.e.,
\be \label{hatZ}
\hat{Z}=\begin{bmatrix}
 \vdots\\
(\hat{z}_{s,t, k})^T \\
 \vdots
\end{bmatrix}.
\ee
Furthermore, we let
\be \label{matZ}
Z \coloneqq
\bbm  \hat{Z} \\ \hat{D}  \ebm
 \quad \text{where} \quad
 \hat{D} = \left[
  \begin{array}{cc} \hat{a}\ & 0  \\  0\ & \hat{b} \end{array}
  \right]^T.
\ee

For the polynomial $\hat{\psi}_{s,t,k}$ defined as in \reff{hat:phi}, it holds
\[
\hat{\psi}_{s,t,k}( \pm \hat{a}, \pm \hat{b} ) = 0.
\]
This means that four pairs $(\pm \hat{a}, \pm \hat{b})$
are common zeros of polynomials in the set
\be \label{hatPhi}
\hat{\Psi}  \coloneqq
\left\{
\hat{\psi}_{s,t,k}(\hat{a}, \hat{b})
\left| \baray{c}
(i_s,\,j_s), (i_t,\,j_t) \in \Omega_k, \\
s <  t, \, k = 1, \ldots,  n_3
\earay \right.
\right\} .
\ee
The relationship between the biquadratic optimization \eqref{pri-find-ab}
and the rank-$1$ tensor completion problem is given as follows.
Recall that $\dt \mA_{\Omega}$ is the noise tensor given in (\ref{eq:mA=hatmA+dt}).
Throughout the paper, by saying a property holds when
$\dt \mA_{\Omega}$ is sufficiently small, we mean that
there exists a small $\varepsilon > 0$ such that the property holds whenever $\Vert \dt \mA \Vert_{\Omega} < \varepsilon$.

\begin{thm}  \label{th-appr-ab}
Let $\widehat{\mA} = \hat{a}\otimes \hat{b}\otimes \hat{c}$
with $\| \hat{a} \| = \| \hat{b} \| = 1$. Assume $\rank\, Z = n_1+n_2$
and $(\pm \hat{a}, \pm \hat{b})$ are the only common zeros of polynomials in $\hat{\Psi}$.
Let $(a^*, b^*)$ be a minimizer of \eqref{pri-find-ab}.
Then, there exists a constant $C_1>0$ depending on
$(\hat{a}, \hat{b}, \hat{c})$ such that
\be \label{bd:a*b*}
\|a^* \otimes b^* -  \theta \hat{a} \otimes \hat{b} \|  \le  C_1  \|\dt \mA \|_{\Omega}
\quad \text{for some} \quad \theta \in \{-1, 1\},
\ee
when $\dt \mA_{\Omega}$ is sufficiently small.
\end{thm}
\begin{proof}
To show the dependence of coefficients of $f(a,b)$ on $\mA_{\Omega}$, we define
\[
F(a,b, \mA_{\Omega}) \, \coloneqq \,
\sum\limits_{ k=1 }^{n_3}  \sum\limits_{1 \le s < t \le m_k}
(\mA_{i_s j_s k} a_{i_t }b_{j_t} -\mA_{i_t j_t k} a_{i_s}b_{ j_s})^2  .
\]
For a given $\mA_{\Omega}$, the objective $f(a,b)$ of \eqref{pri-find-ab}
is the same as $F(a,b, \mA_{\Omega}).$
The Lagrangian function of \eqref{pri-find-ab} is
\[
L(a,b,\lambda_1,\lambda_2)  \coloneqq
F(a,b,\mA_{\Omega}) - \lambda_1 (||a||^2-1)- \lambda_2 (||b||^2-1).
\]
When $\|\mA-\widehat{\mA}\|_{\Omega}=0$,
we have $F(\hat{a},\hat{b},\widehat{\mA}_{\Omega}) = 0$.
Since $F(a,b,\widehat{\mA}_{\Omega})\ge 0$ for all $a,b$, we further have
\[
\nabla_{(a,b)}F(\hat{a},\hat{b},\widehat{\mA}_{\Omega}) \, =  \,  0.
\]
Since both $\hat{a}$ and $\hat{b}$ are nonzero vectors,
the linear independence constraint qualification condition
(see \cite[Section~5.1]{NiePoly}) holds at $(\hat{a},\hat{b})$,
thus $\hat{\lambda}_1=\hat{\lambda}_2=0$
is the unique solution to the following
{\it Karush-Kuhn-Tucker} system about $(\lmd_1, \lmd_2)$
\begin{equation}\label{max-ab}
 \begin{bmatrix}
 \nabla_{a} F(\hat{a},\hat{b},\widehat{\mA}_{\Omega}) - 2 \lambda_1\hat{a} \\
 \nabla_{b} F(\hat{a},\hat{b},\widehat{\mA}_{\Omega}) - 2 \lambda_2 \hat{b}
\end{bmatrix} = 0  .
\end{equation}
Let $\Lambda$ be the function such that
\[
\Lambda(a,b,\lambda_1,\lambda_2,{\mA}_{\Omega})\, \coloneqq
\nabla L(a,b,\lambda_1,\lambda_2) =
\begin{bmatrix}
\nabla_{a} F(a,b,\mA_{\Omega})-2\lambda_1 a \\
\nabla_{b} F(a,b,\mA_{\Omega})-2\lambda_2 b \\
 1-a^Ta\\
 1-b^Tb
\end{bmatrix}.
\]
Note that
\[
\Lambda(\hat{a},\hat{b},0,0,\widehat{\mA}_{\Omega})=0.
\]
Its Jacobian matrix at $(\hat{a},\hat{b},0,0,\widehat{\mA}_{\Omega})$ is
\[
J =
\begin{bmatrix}
  H(\hat{a},\hat{b}) &  \hat{D}^T  \\
  \hat{D} & 0
\end{bmatrix} ,
\]
where $H$ and $\hat{D}$ are given as in \reff{eq:Hab} and \reff{matZ}, respectively.
We show that $J$ is nonsingular. Suppose there exist
\[
p = \bbm u \\ v \ebm   \in\re^{n_1+n_2}, \quad h  \in\re^{2}, \quad
  H(\hat{a},\hat{b})p+ \hat{D}^T h =0,  \quad \hat{D} p =  0,
\]
which implies
\[
p^T H(\hat{a},\hat{b})p + p^T \hat{D}^T h  = p^T H(\hat{a},\hat{b})p  =0.
\]
Moreover, by (\ref{eq:Hab}), we have
\[
0 = p^T H(\hat{a},\hat{b}) p  =   \sum_k\sum_{s,t}
   \left| (\hat{z}_{s,t,k} )^T p\right|^2 \ge  0.
\]
So $(\hat{z}_{s,t,k} )^T p = 0$ for all indices $s,t,k$ in the range.
Since $\hat{D}p=0$, we know $\hat{a}^Tu = \hat{b}^T v = 0$,
hence $Zp = 0$. Since $Z$ has full column rank, we get $p=0$.
This shows that $J$ is nonsingular.

By the Implicit Function Theorem (see \cite[Theorem~9.28]{rudin1976}),
there exist an open neighbourhood $\mc{V}$ of $(\hat{a},\hat{b},\widehat{\mA}_{\Omega})$
and a continuously differentiable function
\[
\varphi = (\varphi_1, \varphi_2, \varphi_3, \varphi_4):\,
\mc{U} \to  \mc{V},
\]
defined in an open neighbourhood $\mc{U}$ of $\widehat{\mA}_{\Omega}$, such that
\[
\varphi_1(\widehat{\mA}_{\Omega}) = \hat{a},\quad
\varphi_2(\widehat{\mA}_{\Omega}) = \hat{b},\quad
\varphi_3(\widehat{\mA}_{\Omega}) = 0, \quad
\varphi_4(\widehat{\mA}_{\Omega}) = 0,
\]
\[
\Lambda(\varphi_1(\mA_{\Omega}),\,\varphi_2(\mA_{\Omega}),\,
\varphi_3(\mA_{\Omega}),\,\varphi_4(\mA_{\Omega}),\,\mA_{\Omega}) = 0
\quad \text{for all} \quad  \mA_{\Omega}\in \mc{U},
\]
and $\varphi(\mA_{\Omega})$ is the unique solution to $\Lambda(a,b,\lambda_1,\lambda_2,{\mA}_{\Omega}) = 0$
such that 
\[
(\varphi(\mA_{\Omega}),\mA_{\Omega})\in\mc{V}
\quad \text{for every} \quad \mA_{\Omega}\in \mc{U}.
\]
Since $\varphi$ is continuously differentiable, there exist scalars $\varepsilon_0>0$ and $C_0>0$ such that $\mc{U}_0\coloneqq \{ \mA_{\Omega} : \Vert\mA - \widehat{\mA}\Vert_{\Omega} < \varepsilon_0 \} \subseteq \mc{U}$,
and for all $\mA_{\Omega}\in \mc{U}_0$ it holds that
\be  \label{gdA<=C0dA}
\| \big( \varphi_1(\mA_{\Omega}), \varphi_2(\mA_{\Omega}) \big) -
   \big( \varphi_1(\widehat{\mA}_{\Omega}), \varphi_2(\widehat{\mA}_{\Omega}) \big)
   \|  \le C_0 \|  \dt \mA_{\Omega} \|.
\ee

Note that
$
F(\pm \hat{a},  \pm \hat{b}, \widehat{\mA}_{\Omega}) =  0.
$
By the assumption, $(\pm \hat{a}, \pm \hat{b})$ are the only global minimizers of
the function $F(a,b,\widehat{\mA}_{\Omega})$ on the set
\[  S = \{(a,b): \| a \| = \| b \| = 1 \}. \]
For each $\dt >0$, we have
\[
F(a, b, \widehat{\mA}_{\Omega}) \, > \, 0  \quad \text{for all} \quad
(a,b) \in S \setminus N_{\dt},
\]
where
\[
N_\dt  \coloneqq
\bigcup_{ (\theta_1,  \theta_2) \in \{ (\pm 1, \pm 1 ) \} }
\{ (a, b):  \big(a- \theta_1  \hat{a} \big)^2  +
\big(b - \theta_2  \hat{b} \big)^2 < \dt \}.
\]
Since $S \setminus N_{\dt}$ is compact, there exists $\eps >0$ depending on $\dt$, such that
\[
F(a, b, \widehat{\mA}_{\Omega}) \, \ge  \, \eps  \quad \text{for all} \quad
(a,b) \in S \setminus N_{\dt} .
\]
Note that $\dt \mA_{\Omega} = \mA_{\Omega} - \widehat{\mA}_{\Omega}$.
Since $F$ is continuous, there exists $\varepsilon_1 > 0$
such that for all $\Vert \dt \mA\Vert_{\Omega} < \varepsilon_1$, it holds
\be \label{F>=eps/2}
F(a, b,  \mA_{\Omega}) \, \ge  \, \eps/2  \quad \text{for all} \quad
(a,b) \in S \setminus N_{\dt} .
\ee
Note that for every $\mA_{\Omega}\in \mc{U}$, $\varphi(\mA_{\Omega})$ is the only solution to $\Lambda(a,b,\lambda_1,\lambda_2,{\mA}_{\Omega}) = 0$ such that $(\varphi(\mA_{\Omega}),\mA_{\Omega})\in\mc{V}$.
Thus, we fix $\dt > 0$ small enough so that for every $\mA_{\Omega} \in \mc{U}$, the pair
$(\varphi_1(\mA_{\Omega}), \varphi_2(\mA_{\Omega}))$ is the only critical point
of  \eqref{pri-find-ab} in the open ball given by
\[
(a - \hat{a})^2 + (b - \hat{b})^2   <   \dt.
\]
Since for every $(\theta_1, \theta_2) \in \{ (\pm 1, \pm 1) \}$, it holds
\[
\begin{gathered}
\nabla_{a}F(\theta_1a,\theta_2b,\mA_{\Omega}) = \theta_1\nabla_{a}F(a,b,\mA_{\Omega}) ,\\
\nabla_{b}F(\theta_1a,\theta_2b,\mA_{\Omega}) = \theta_2\nabla_{b}F(a,b,\mA_{\Omega}),
\end{gathered}
\]
we have $\Lambda(a,b,\lambda_1,\lambda_2,{\mA}_{\Omega}) = 0$ if and only if
$\Lambda( \pm a, \pm b,\lambda_1,\lambda_2,{\mA}_{\Omega}) = 0$.
So, for each pair $(\theta_1, \theta_2) \in \{ (\pm 1, \pm 1) \}$,
$( \theta_1 \varphi_1(\mA_{\Omega}), \theta_1 \varphi_2(\mA_{\Omega}))$
is the only critical point of (\ref{pri-find-ab}) in the open ball
\[
(a - \theta_1\hat{a})^2 + (b - \theta_2\hat{b})^2   <   \dt.
\]
By the continuity of $F$ and $\varphi$, there exists $\varepsilon_2<\varepsilon_1$ such that for all $\mc{A}_{\Omega}$ satisfying $\Vert \dt \mA\Vert _{\Omega} < \varepsilon_2$, it holds
\[
\begin{gathered}
F(\varphi_1(\mA_{\Omega}), \varphi_2(\mA_{\Omega}),  \mA_{\Omega}) \, \le  \, \eps/3, \\
\big(\varphi_1(\mA_{\Omega}) - \hat{a} \big)^2 +  \big( \varphi_2(\mA_{\Omega}) - \hat{b} \big)^2 < \dt.	
\end{gathered}
\]
This further implies the following holds for all $(\theta_1, \theta_2) \in \{ (\pm 1, \pm 1) \}$:
\be  \label{F<=eps/3}
\begin{gathered}
F( \theta_1 \varphi_1(\mA_{\Omega}), \theta_2 \varphi_2(\mA_{\Omega}),  \mA_{\Omega}) \, \le  \, \eps/3 ,   \\
\big(\theta_1 \varphi_1(\mA_{\Omega}), \theta_2 \varphi_2(\mA_{\Omega}) \big) \in S \cap N_\dt.
\end{gathered}
\ee

Since the feasible set of (\ref{pri-find-ab}) is compact,
it must have global minimizers for all $\mc{A}_{\Omega}$.
Thus, by (\ref{F>=eps/2}) and (\ref{F<=eps/3}),
the global minimizers of (\ref{pri-find-ab}) exist,
and they all belong to $N_{\dt}$, when $\Vert \dt \mA\Vert_{\Omega} < \varepsilon_2<\varepsilon_1$.
Moreover, the only critical points of (\ref{pri-find-ab}) in $N_{\dt}$
are $( \pm \varphi_1(\mA_{\Omega}), \pm \varphi_2(\mA_{\Omega}))$, and
\[ F( a,  b, \mA_{\Omega}) = F(\pm a, \pm b,  \mA_{\Omega}). \]
Therefore, when $\dt \mA_{\Omega}$ is sufficiently small,
$(\pm \varphi_1(\mA_{\Omega}), \pm \varphi_2(\mA_{\Omega}))$
are the only global minimizers of \eqref{pri-find-ab},
and $(a^*, b^*)$ must be one of these four pairs.
Since $\varphi_1(\widehat{\mA}_{\Omega}) = \hat{a}$
and $\varphi_2(\widehat{\mA}_{\Omega}) = \hat{b}$,
the tensor product $a^* \otimes b^*$ must be sufficiently close to
either $\hat{a} \otimes \hat{b}$ or  $-\hat{a} \otimes \hat{b}$.
Therefore, for all $\mA_{\Omega}$ such that $\Vert \dt \mA\Vert_{\Omega} < \min(\varepsilon_0,\varepsilon_2)$,
\reff{gdA<=C0dA} implies that there is a constant $C_1>0$
such that \reff{bd:a*b*} holds for $\theta = 1$ or $\theta = -1$.
Moreover, the constant $C_1$ only depends on $(\hat{a}, \hat{b}, \hat{c})$.
\end{proof}

\begin{rmk}\label{rmk:assumptions}\rm
The following remarks provide further expositions
for Theorem~\ref{th-appr-ab} and its assumptions.
\bit

\item [(i)] Theorem~\ref{th-appr-ab} shows the existence of $\varepsilon$
such that \reff{bd:a*b*} holds. However, the theorem does not provide
an explicit estimation of such $\varepsilon$. This is interesting for future work.

\item [(ii)] The assumption that $(\pm \hat{a}, \pm \hat{b})$
are the only common zeros of $\hat{\Psi}$ does not necessarily imply that they are the only global minimizers of (\ref{pri-find-ab}).
It concerns the uniqueness of common solutions (up to switching of signs)
for $\hat{\Psi}$, which is given by the noise-free tensor $\widehat{\mc{A}}$
instead of the observed tensor $\mc{A}$ that defines (\ref{pri-find-ab}).
If there exists a pair of indices $(i,j)$ such that $(i,j,k)\notin \Omega$ for all $k$,
then this assumption typically does not hold.
For such a case, one can check that $(a,b) = (e_i,e_j)$
is a global minimizer of (\ref{pri-find-ab})
since $f(e_i,e_j) = 0$, and thus solving (\ref{pri-find-ab})
may yield a minimizer $e_i\otimes e_j$ that is far away from
both $\pm \hat{a}\otimes \hat{b}$.

\item [(iii)] Our algorithm is applicable,
regardless of whether the assumptions of Theorem~\ref{th-appr-ab} hold or not.
Even for the case where there exists $(i,j)$ such that $(i,j,k)\notin \Omega$ for all $k$,
one may heuristically estimate some entry $\mc{A}_{ijk}$ and add the triple $(i,j,k)$ to $\Omega$ before solving (\ref{pri-find-ab}). This can prevent the algorithm from falling to the unwanted trivial solution $(e_i,e_j)$. We refer to Subsection~\ref{sc:application} for this.

\eit
\end{rmk}

\subsection{Retrieval of rank-$1$ tensor completion}

Suppose that $(a^*,b^*)$ is an optimizer pair of $\eqref{pri-find-ab}$.
We look for the vector $c^*\in \re^{n_3}$ such that the distance
$\| \mc{A} - a^*\otimes b^*\otimes c^*\|_{\Omega}$ is minimum.
This requires to solve the linear least square problem:
\be  \label{eq-find-c}
\min_{c\in\re^{n_3}}  \quad  g(c)  \coloneqq
\sum_{(i,j,k)\in\Omega}(\mathcal{A}_{ijk}-a^*_i b^*_j c_k)^2.
\ee
We remark that even if the perturbation $\dt \mA$ is small,
a minimizer $c^*$ of \reff{eq-find-c} may not give a satisfying rank-$1$
tensor completion $a^*\otimes b^*\otimes c^*$, i.e.,
the distance $\| \mc{A} - a^*\otimes b^*\otimes c^*\|_{\Omega}$ may be large.
The following is such an exposition.

\begin{Example}\label{ex:singular}\rm
Let $\mc{A}\in \re^{3\times 3 \times 3}$ be a partially given tensor with
\[ \Omega = [2]\times [2] \times [3], \quad \mc{A}_{ijk} =
\hat{a}_i\otimes\hat{b}_j\otimes\hat{c}_k + \dt\mc{A}_{ijk},\ (i,j,k)\in\Omega, \]
where each $\hat{a}_i=\hat{b}_j=\hat{c}_k=1$ and the noise $\dt\mc{A}_{ijk} = 10^{-6}$.
Since the variables $a_3$, $b_3$ do not appear in (\ref{pri-find-ab}), one may check that
$a^* = b^* = (0,0,1)$ minimizes (\ref{pri-find-ab}) with the minimum value $f^* = 0$.
Moreover, the objective value of \reff{eq-find-c}
is constantly equal to $12(1 + 10^{-6})^2$ for all $c\in\re^{n_3}$,
since $a^*_i b^*_j c_k = 0$ for every $(i,j,k)\in\Omega$.
Thus, no matter how $c^*$ is selected, we always have
\[
\left \|
 \hat{a} \otimes  \hat{b} \otimes  \hat{c}   -
 a^* \otimes  b^* \otimes  c^*
\right \| \ge 4 \gg \Vert \dt\mc{A}\Vert_{\Omega}.
\]
\end{Example}

The reason for the above case to happen is that the minimizer $(a^*, b^*)$
makes the least square problem \reff{eq-find-c} singular.
This leads to the following definition.

\begin{defi}
A pair $(a,b)$ of nonzero vectors is said to be singular if there exists $k\in [n_3]$
such that $m_k = 0$ or
\[
a_{i_1}b_{j_1} = \cdots = a_{i_{m_k}}b_{j_{m_k}}  = 0 ,
\]
where $\Omega_k = \{ (i_1,j_1)  \ddd  (i_{m_k},j_{m_k}) \}$ as in \reff{Omega_k}.
The pair $(a,b)$ is said to be nonsingular if it is not singular.
\end{defi}

We remark that if $m_k = 0$ for some $k$, then the variable $c_k$
does not appear in \reff{eq-find-c} and we can choose arbitrary values for $c_k$.
Note that $(a,b)$ is singular if and only if
all four pairs $(\pm a, \pm b)$ are singular.
The objective $g(c)$ in \reff{eq-find-c} can be equivalently written as
\be
g(c) =  \sum_{k=1}^{n_3} g_k(c_k) \quad \text{where} \quad
g_k(c_k)  =  \sum_{(i,j) \in \Omega_k}(\mathcal{A}_{ijk}-a^*_i b^*_j c_k)^2.
\ee
When $(a^*,b^*)$ is nonsingular, the minimizer $c_k^*$ of $g_k(c_k)$ is explicitly given as
\be \label{ck*}
c_k^* =  \Big( \sum_{s=1}^{m_k} \mA_{i_s j_s k} a^*_{i_s} b^*_{j_s} \Big) /
  \Big( \sum_{s=1}^{m_k} ( a^*_{i_s} b^*_{j_s} )^2 \Big),
\ee
and the minimizer of \reff{eq-find-c} is
\be \label{cstar}
c^* \, = \, (c_1^* ,  \ldots,   c_{n_3}^*).
\ee

The distance between $\hat{c}$ and $c^*$ can be estimated as follows.

\begin{thm}  \label{th-appr-c}
Let $\widehat{\mA} = \hat{a}\otimes \hat{b}\otimes \hat{c}$
with $\| \hat{a} \| = \| \hat{b} \| = 1$. Assume $\rank\, Z = n_1+n_2$
and $(\pm \hat{a}, \pm \hat{b})$ are the only common zeros of polynomials in $\hat{\Psi}$.
Suppose $(\hat{a}, \hat{b})$ is nonsingular
and $\hat{c}_k \ne 0$ for all $k$.
Let $(a^*, b^*)$ be a minimizer of (\ref{pri-find-ab})
and let $c^*$ be given as in (\ref{cstar}).
There exists a constant $C >0$, depending on $(\hat{a}, \hat{b}, \hat{c})$,  such that
\be \label{eq:c*-cleK}
  \| a^* \otimes b^* \otimes c^* -  \hat{a} \otimes \hat{b} \otimes \hat{c} \|
\le   C  \| \dt \mA \|_{ \Omega },
\ee
when the perturbation $\dt \mA$ is sufficiently small.
\end{thm}

%%%%%%%%%%%%%%%%%%%%%%%%%%%%%%%%%%%%
\begin{proof}
Note that $(a^*, b^*)$ is nonsingular if and only if all the four pairs
$(\pm a^*, \pm b^*)$ are nonsingular.
Since $(\hat{a}, \hat{b})$ is nonsingular for (\ref{eq-find-c}),
$(a^*, b^*)$ is also nonsingular when $\dt\mA$ is sufficiently small,
by Theorem~\ref{th-appr-ab}. So, $c^*$ is the unique optimizer
for (\ref{eq-find-c}), given in \reff{cstar}.
For convenience of analysis, we can generally assume
\reff{bd:a*b*} holds for $\theta = 1$ in Theorem~\ref{th-appr-ab};
otherwise one can replace $a^*$ by $-a^*$.

For a fixed $k  \in  [n_3]$, denote the vectors
\[
\begin{array}{ccl}
w^*  &  \coloneqq  & \bbm a^*_{i_1}b^*_{j_1} & a^*_{i_2}b^*_{j_2} & \cdots &
      a^*_{i_{m_k}}b^*_{j_{m_k}} \ebm^T,  \\
\hat{w}  &  \coloneqq  &  \bbm \hat{a}_{i_1}\hat{b}_{j_1} & \hat{a}_{i_2}\hat{b}_{j_2} & \cdots &
      \hat{a}_{i_{m_k}}\hat{b}_{j_{m_k}} \ebm^T,  \\
v &  \coloneqq  & \bbm \mathcal{A}_{i_1j_1k} & \mathcal{A}_{i_2j_2k} & \cdots &
       \mathcal{A}_{i_{m_k}j_{m_k}k} \ebm^T,  \\
\hat{v}  & \coloneqq &  \bbm \widehat{\mA}_{i_1j_1k} & \widehat{\mA}_{i_2j_2k} & \cdots &
        \widehat{\mA}_{i_{m_k}j_{m_k}k} \ebm^T .
\end{array}
\]
By the given assumption, we have
\[
|\hat{w}^T\hat{v}| > 0,  \quad  \| \hat{w} \|^2 > 0 .
\]
Let $\delta_w \coloneqq w^* - \hat{w}$ and $\delta_v  \coloneqq  v - \hat{v}$, then
\[
\Vert \dt_v\Vert  \le  \Vert \dt \mA \Vert_{\Omega} .
\]
By Theorem~\ref{th-appr-ab}, there exists a constant $D_k$ such that
\be  \label{eq:dtw<Ddtv}
\Vert\delta _w  \Vert    \le    D_k \Vert \dt \mA \Vert_{\Omega} ,
\ee
when $\dt \mA$ is sufficiently small.
Note that $\hat{c}_k =\frac{\hat{w}^T\hat{v}}{\hat{w}^T\hat{w}}$.
Thus we have
\[
\begin{aligned}
c^*_k - \hat{c}_k &=\frac{(w^*)^Tv}{ \| w^* \|^2 } - \frac{\hat{w}^T\hat{v}}{ \| \hat{w} \|^2 } = \frac{(\hat{w}+\delta_w)^T(\hat{v}+\delta_v)}{(\hat{w}+\delta_w)^T(\hat{w}+\delta_w)} - \frac{\hat{w}^T\hat{v}}{\hat{w}^T\hat{w}}\\
& = \hat{c}_k  \frac{ \frac{\dt_w^T\hat{v} + \dt_v^T\hat{w} + \dt_w^T\dt_v}{\hat{w}^T\hat{v}} - \frac{2\hat{w}^T \dt_w + \dt_w^T\dt_w}{ \| \hat{w} \|^2 } }
{1+ \frac{2\hat{w}^T \dt_w + \dt_w^T\dt_w}{ \| \hat{w} \|^2}  }.
\end{aligned}
\]
Since $\Vert\dt_w\Vert  \le D_k  \Vert  \dt \mA \Vert_{\Omega}$,
for small $\Vert  \dt  \mA  \Vert_{\Omega}$, it holds
\[ \nn \label{eq:wdtw<1/2}
\frac{2\hat{w}^T \dt_w + \dt_w^T\dt_w}{\hat{w}^T\hat{w}} \ge -\frac{1}{2}.
\]
So we can get
\[
\frac{ |c^*_k - \hat{c}_k| }{ 2|\hat{c}_k| }   \le
 \left|  \frac{\dt_w^T\hat{v} + \dt_v^T\hat{w} + \dt_w^T\dt_v}{\hat{w}^T\hat{v}} -
 \frac{2\hat{w}^T \dt_w + \dt_w^T\dt_w}{ \| \hat{w} \|^2 }  \right|   .
\]
By the Cauchy-Schwartz inequality, we have
\[
\frac{|c^*_k - \hat{c}_k| }{  2|\hat{c}_k| } \le
    \frac{\| \dt_w \| \| \hat{v} \| + \|\dt_v \| \| \hat{w} \|  + \| \dt_w \| \| \dt_v \|}{ |\hat{w}^T\hat{v}| } +
        \frac{2 \| \hat{w} \|  \| \dt_w \|  + \| \dt_w \|^2 }{ \| \hat{w} \|^2 }    .
\]
By the given assumption, there exists a constant $K_k > 0$ such that
\[
|c^*_k - \hat{c}_k|  \le   K_k \Vert  \dt \mA  \Vert_{\Omega}.
\]
Let $C_2 \coloneqq \sqrt{ n_3 } \max_k K_k$, then
\[
\Vert c^*-\hat{c}\Vert  = \Big( \sum_{k=1}^{n_3} |c^*_k - \hat{c}_k|^2 \Big)^{1/2}
 \le  C_2  \Vert  \dt \mA  \Vert_{\Omega}.
\]

Next, we note that
\[
\baray{rl}
    &  \| a^* \otimes b^* \otimes c^* -  \hat{a} \otimes \hat{b} \otimes \hat{c} \|  \\
\le &  \| a^* \otimes b^* \otimes c^* -  \hat{a} \otimes \hat{b} \otimes c^* \|  +
\|  \hat{a} \otimes \hat{b} \otimes c^* -  \hat{a} \otimes \hat{b} \otimes \hat{c} \|  \\
= &  \| a^* \otimes b^*   -  \hat{a} \otimes \hat{b} \| \|  c^* \|  +
\|  \hat{a} \otimes \hat{b} \| \|   c^* -    \hat{c} \|  \\
\le &  C  \Vert  \dt \mA  \Vert_{\Omega}
\earay
\]
for some constant $C>0$. So \reff{eq:c*-cleK} holds.
\end{proof}

Summarizing the above, we propose the following algorithm to
get rank-$1$ tensor completions.

\begin{alg}  \label{alg:r1tc}  \rm
For a partially given tensor $\mc{A}$, let $S\coloneq \emptyset$ and do:
\begin{itemize}

\item[Step~1] For each $k=1\ddd n_3$, obtain the index set $\Omega_k$.
Then, formulate the set $\Phi$ of biquadratic polynomials as in (\ref{polyset:Phi}).

\item[Step~2] Solve the biquadratic optimization (\ref{pri-find-ab})
for a set $U$ of some minimizers.

\item[Step~3] For each $(a^*,b^*)\in U$, if it is nonsingular,
then apply the formula (\ref{ck*})-(\ref{cstar}) to get $c^*$,
and update $S\coloneq S\cup \{ a^*\otimes b^*\otimes c^* \}$.

\item[Step~4] Output the set $S$ of rank-$1$ completions of $\mA$, and stop.

\end{itemize}	
\end{alg}

The major computational cost of Algorithm~\ref{alg:r1tc} is to solve the biquadratic optimization (\ref{pri-find-ab}).
In the worst case, finding global minimizers for biquadratic optimization problems is NP-hard; see \cite[Theorem~2.2]{Ling2010}.
In Section~\ref{sc:minimizer}, we propose a convex relaxation method for solving (\ref{pri-find-ab}).
For the partially given tensor $\mc{A} \in \re^{n_1 \times n_2  \times n_3}$,
this method solves the semidefinite program (\ref{moment-for-ab}) with PSD matrix variable of size $n_1n_2$-by-$n_1n_2$.
In our computational experiments, this relaxation often finds global minimizers of (\ref{pri-find-ab}).
This is demonstrated in Section~\ref{sc:minimizer}.
Semidefinite programs can be solved efficiently
by interior-point methods. We refer to \cite{Todd01}
for the complexity of solving semidefinite programs.

%%%%%%%%%%%%%%%%%%%%%%%%%%%%%%%%%%%%%%%%%%%%%%%%%%%%%%%%%%%%%%
\section{The convex relaxation}
\label{sc:minimizer}

In this section, we propose a convex relaxation method for
solving the biquadratic optimization \reff{pri-find-ab}.

Recall the index set $\Gamma$ given in \reff{df:label:ijkl}.
We observe the expansion
\[
(aa^T) \otimes (bb^T) = \sum_{(i,j,k,l)\in\Gamma } a_i a_j b_k b_l \, E_{ij} \otimes E_{kl} ,
\]
where $E_{ij}, E_{kl}$ are the basis matrices as in \reff{basis:EijEkl}.
Recall that $\re^{\Gamma}$ denotes the space of vectors whose entries are labelled by
$(i,j,k,l) \in \Gamma$ and the linear symmetric matrix function $G[y]$ in
$y  \in \re^{\Gamma}$ is defined in \reff{eq:Gy}:
\[
G[y] \, = \,  \sum_{(i,j,k,l)\in\Gamma }y_{ijkl}  E_{ij} \otimes E_{kl}.
\]
It is worthy to note that there exist vectors $a = (a_i)$ and $b = (b_k)$ such that
\be \label{yijkl=aabb}
y_{ijkl} = a_i a_j b_k b_l \quad \text{for all} \quad   (i,j,k,l)\in\Gamma
\ee
if and only if $G[y] = (a a^T) \otimes  (bb^T)$.
By the expansion of $(aa^T) \otimes (bb^T)$ as in \reff{eq:ababT},
we can write $G[y]$ in the block form
\be   \label{mom-block}
\begin{array}{cl}
 G\left [  y \right ] &=\begin{bmatrix}
 B_{11}  & B_{12} & \cdots & B_{1n_1} \\
 B_{21}  & B_{22} & \cdots & B_{2n_1} \\
 \vdots & \vdots   &\ddots   & \vdots \\
 B_{n_11}  & B_{n_12} & \cdots  & B_{n_1n_1}
\end{bmatrix} ,
\end{array}
\ee
where each block $B_{ij}$ is also a $n_2$-by-$n_2$ symmetric matrix.
Note that $G[y]$ itself is symmetric.

We expand the biquadratic objective $f(a,b)$ of \reff{pri-find-ab} as
\[
f(a,b) = \sum\limits_{(i,j,k,l)\in\Gamma } f_{ijkl}a_i a_j b_k b_l.
\]
For $y  \in \re^{\Gamma}$, we define the bilinear function
\[
\langle f,y \rangle  \,  \coloneqq  \,  \sum\limits_{(i,j,k,l)\in\Gamma } f_{ijkl}y_{ijkl}.
\]
If \reff{yijkl=aabb} holds for $\|a\| = \| b \|=1$, then $f(a,b) = \langle f,  y \rangle$ and
\[
\trace\,G[y] =  \trace \, (aa^T) \otimes (bb^T) =    (a^Ta) (b^T b)  = 1.
\]
Therefore, the biquadratic optimization \reff{pri-find-ab} is equivalent to
\be   \label{moment_nonlinear}
\left\{
\begin{array}{cl}
  \min & \langle f, y \rangle \\
  \st   & \trace \, G[y]  = 1,  \\
       &  G[ y ] = (a a^T) \otimes  (bb^T)  \succeq 0. \\
\end{array}
\right.
\ee
If $G[y] = (a a^T) \otimes  (bb^T)$ is ignored,
then we get the convex relaxation
\be    \label{moment-for-ab}
\left\{
\begin{array}{cl}
\displaystyle \min_{y\in\re^{\Gamma}} & \langle f, y \rangle \\
\st  & \trace \, G[y] =  1,  \\
       &  G\left[ y\right]\succeq 0. \\
\end{array}
\right.
\ee
Its dual optimization problem is
\begin{equation}\label{sos-for-ab}
\left\{
\begin{array}{cl}
 \mbox{max} &\gamma\\
\st & f(a, b) - \gamma \cdot \left \| a\otimes b \right \| ^2 \in \mc{Q},
\end{array}
\right.
\end{equation}
where $\mc{Q}$ is the cone of sums of squares of bilinear forms in $(a,b)$:
\[
\mc{Q} \, \coloneqq \, \left\{(a\otimes b)^T Y(a\otimes b):\,
Y\in\mathcal{S}_+^{n_1n_2}
\right\}.
\]
The following is a basic property for the relaxation pair
\reff{moment-for-ab}-\reff{sos-for-ab}.

\begin{prop}\label{prop:relx}
Denote by $f^*, p^*, d^*$ the optimal values for (\ref{pri-find-ab}),
(\ref{moment-for-ab}) and (\ref{sos-for-ab}), respectively.
Then, both $d^*$ and $p^*$ are attainable, and
\[ d^*= p^*\le f^*. \]
\end{prop}
\begin{proof}
Let $(a,b)$ be a feasible point for (\ref{pri-find-ab}),
and let $y\in\re^{\Gamma}$ be such that \reff{yijkl=aabb} holds.
Then $y$ is feasible for (\ref{moment_nonlinear}),
so $y$ must also be feasible for (\ref{moment-for-ab}).
The minimum value of (\ref{moment_nonlinear}) is $f^*$, so $p^*\le f^*$
since \reff{moment-for-ab} is a relaxation of \reff{moment_nonlinear}.

Let $\tilde{y}\in\re^{\Gamma}$ be such that $\tilde{y}_{ijkl} =
\frac{1}{n_1n_2}$ if $i=j$ and $k=l$, and $\tilde{y}_{ijkl} = 0$ otherwise.
Then we have
$
\mathrm{Trace}\, G\left[ \tilde{y}\right]=1
$
and $G\left[ \tilde{y}\right]\succ 0$, because
\[
G\left[ \tilde{y}\right] \, = \,
\Big( \int_{\|a\| = 1} aa^T  \mt{d} \mu_1 \Big)  \otimes
\Big( \int_{\|b\|= 1} bb^T  \mt{d} \mu_2 \Big),
\]
where $\mu_1$ and $\mu_2$ denote the uniformly distributed probability measure
on the unit spheres of $\|a\|=1$ and $\|b\|=1$ respectively.
So $\tilde{y}$ is a strictly feasible point, then we know
$d^* = p^*$ and $d^*$ is attainable,
by the strong duality theorem (see \cite[Sec.~1.5]{NiePoly}).

Next, we see that
\[\mA_{i_sj_sk}a_{i_t}b_{j_t} -
\mA_{i_tj_tk}a_{i_s}b_{j_s} =
(a\otimes b)^T(\mathcal{A}_{i_sj_sk}e_{i_t}\otimes e_{j_t}-
\mathcal{A}_{i_tj_tk}e_{i_s}\otimes e_{j_s}),
\]
so there exists $B \in \mc{S}^{n_1n_2}$ such that
\[f(a,b) = (a\otimes b)^T B (a\otimes b).\]
If $\tilde{\gm}$ is negatively big enough, then $B - \tilde{\gm} I_{n_1n_2}\succ 0$ and
\[
f-\tilde{\gm}\cdot \left \| a\otimes b \right \| ^2 =
(a\otimes b)^T (B - \tilde{\gm} I_{n_1n_2}) (a\otimes b)^T \in \mc{Q}.
\]
This implies that (\ref{sos-for-ab}) is also strictly feasible,
so the optimal value $p^*$ is attainable, by the strong duality theory.
\end{proof}

\begin{rmk}\rm
   To solve the biquadratic optimization~(\ref{pri-find-ab}), one may consider to
   apply the classical Moment-SOS relaxations. However, this is not computationally efficient.
   Even for the lowest order (i.e., two), the Moment-SOS relaxation
   has a positive semidefinite (PSD) matrix variable with length equal to $\binom{n_1+n_2+2}{2}$.
   In contrast, the relaxation (\ref{moment-for-ab}) has one PSD matrix variable
   whose length is only $n_1n_2$, which is much smaller.
\end{rmk}

Suppose $y^*$ is a minimizer of \reff{moment-for-ab}.
A specially interesting case is $\rank\, G[y^*] = 1$.
By Proposition~\ref{prop:rank1}, there must exist a pair $(a^*,b^*)$ such that
\be \label{G[y*]=aa*bb*}
G[y^*] = (a^* {a^*}^T) \otimes  (b^* {b^*}^T), \quad \| a^* \| = \| b^* \| = 1.
\ee
For this case, the relaxation \reff{moment-for-ab} is tight
and the pair $(a^*,b^*)$ must be a minimizer of (\ref{pri-find-ab}).
This can be implied as a special case of Theorem~\ref{th:exact:minimizer}.
The above $a^*,b^*$ can be obtained as follows.
Write $G[y^*]$ in the block form \reff{mom-block}.
Find the indices $i^*, j^*$ such that the block
$B_{i^*i^*} \ne 0$ and the $j^*$th column of $B_{i^*i^*}$ is nonzero.
Then the vectors $a^*,b^*$ can be selected as
\be  \label{eq:astarbstar}
\boxed{
\baray{lcl}
\tilde{a}  &=& \big( (B_{11})_{j^*  j^*}, \ldots, (B_{n_11})_{j^*  j^*} \big), \\
       a^* &=&  \tilde{a}/ \| \tilde{a} \|, \\
\tilde{b}  &=& \big( (B_{i^*i^*})_{1  j^*}, \ldots, (B_{i^*i^*})_{n_2  j^*} \big), \\
       b^* &=& \tilde{b}/ \| \tilde{b} \|.
\earay }
\ee
They must satisfy \reff{G[y*]=aa*bb*} if $\rank\, G[y^*] = 1$.

As shown in Proposition~\ref{prop:relx}, both (\ref{moment-for-ab}) and its dual problem (\ref{sos-for-ab}) are strictly feasible.
Therefore, the semidefinite program (\ref{moment-for-ab}) can be solved in
polynomial time by interior point methods.
We refer to \cite{Todd01} for this.
At the minimizer $y^*$, if $\rank\, G[y^*] = 1$, then one can recover the global minimizer of (\ref{pri-find-ab}) by (\ref{eq:astarbstar}).
In our computational experiments, we often get a rank-$1$ minimizer matrix $ G[y^*]$.
This is empirically demonstrated in Example~\ref{ex:rank_vs_den}.
For these cases, the biquadratic optimization is globally solved
by the relaxation (\ref{moment-for-ab}).

For the case $\rank\, G[y^*] > 1$, the relaxation \reff{moment-for-ab}
is also tight and we can get several minimizers of \reff{pri-find-ab}
if the matrix $G[y^*]$ is separable.
This will be discussed in the next subsection.

\subsection{The case $G[y^*]$ is separable}
\label{sc:sep_Gy}

Recall that the matrix $G[y^*]$ is separable if there exist vectors
$a^{(i)}\in\re^{n_1}$, $b^{(i)}\in\re^{n_2}$ and scalars
$\lmd_i \in \re^1$ such that
\be  \label{eq:sep_decomp}
\boxed{
\baray{c}
G[y^*] = \sum\limits_{i=1}^r \lmd_i
( a^{(i)} {a^{(i)} }^T )  \otimes  ( b^{(i)} {b^{(i)} }^T ) , \\
\Vert a^{(i)}\Vert = \Vert b^{(i)}\Vert = 1,  \,  \lmd_i > 0, \, i=1, \ldots, r .
\earay
}
\ee
We refer to the subsection~\ref{sc:pre_sep} for separable matrices.
The following shows how to get minimizers for \eqref{pri-find-ab} from $G[y^*]$
when it is separable.

\begin{thm}  \label{th:exact:minimizer}
Suppose $y^*$ is a minimizer of the relaxation \reff{moment-for-ab}.
If the decomposition \reff{eq:sep_decomp} holds, then
each pair $(a^{(i)},\, b^{(i)})$ is a minimizer of \eqref{pri-find-ab}.
\end{thm}
\begin{proof}
Denote by $f^*$ and $p^*$ the optimal values of \eqref{pri-find-ab},
\eqref{moment-for-ab} respectively.
By (\ref{eq:sep_decomp}) and the equality constraint in (\ref{moment-for-ab}), we have
\[\begin{aligned}
1 = \trace \, G[y^*]  = & \sum_{s=1}^r \lmd^{(s)} \trace \,
    ( a^{(s)} {a^{(s)} }^T )  \otimes  ( b^{(s)} {b^{(s)} }^T )   \\
  = & \sum_{s=1}^r \lmd_s \Vert a^{(s)}\Vert^2 \Vert b^{(s)}\Vert^2 .
\end{aligned}
\]
Since each $\Vert a^{(s)}\Vert = \Vert b^{(s)}\Vert = 1$, it holds
$\lmd_1 + \cdots + \lmd_r = 1.$
Furthermore,  all $(a^{(i)},b^{(i)})$ are feasible for (\ref{pri-find-ab}), so
\[
 f(a^{(s)},b^{(s)}) \ge  f^*, \quad
 \sum\limits_{s=1}^{r} \lmd_s   f(a^{(s)},b^{(s)}) \ge
 \sum\limits_{s=1}^{r} \lmd_s   f^*   =   f^*.
\]
On the other hand, we also have
\[
\begin{gathered}
 \sum\limits_{s=1}^{r} \lmd_s   f(a^{(s)},b^{(s)}) =
 \sum\limits_{s=1}^{r} \sum\limits_{(i,j,k,l)\in\Gamma }
  f_{ijkl}  \lmd_s a^{(s)}_i a^{(s)}_j b^{(s)}_kb^{(s)}_l \\
=
\sum\limits_{(i,j,k,l)\in\Gamma } f_{ijkl}\sum\limits_{s=1}^{r}
     \lmd_s a^{(s)}_i a^{(s)}_j b^{(s)}_kb^{(s)}_l
=  \sum\limits_{(i,j,k,l)\in\Gamma } f_{ijkl}y^*_{ijkl}  = p^*.	
\end{gathered}
\]
Since  $p^* \le f^*$ by Proposition~\ref{prop:relx}, we get
\[
 \sum\limits_{s=1}^{r} \lmd_s  ( f(a^{(s)},b^{(s)}) - f^* ) \le 0 .
\]
All scalars $\lmd_s > 0$,  so it holds
\[
 f(a^{(i)},b^{(i)}) = f^*,\quad i=1\ddd r,
\]
which implies that every $(a^{(i)},b^{(i)})$ is a minimizer of (\ref{pri-find-ab}).
\end{proof}

The matrix $G[y^*]$ is always separable when $\rank\, G[y^*] = 1$.
But it may or may not be separable if $\rank\, G[y^*] > 1$.
As mentioned in Section~\ref{sc:pre_sep},
we can check if $G[y^*]$ is separable or not by
solving a hierarchy of semidefinite relaxations.
When $G[y^*]$ is not separable, that method can detect non-separability.
When $G[y^*]$ is separable, this method can get the vectors
$(a^{(i)},b^{(i)})$ satisfying (\ref{eq:sep_decomp}).
We refer to \cite{NieZhang16} for how to check separable matrices.

\subsection{The case that $G[y^*]$ is not separable}
\label{sc:notsep}

When $G[y^*]$ is not separable,  how to get minimizers of \reff{pri-find-ab}
becomes very difficult. A typical reason is that
the relaxation \reff{moment-for-ab} may not be tight.
For this case, we may need to solve  (\ref{pri-find-ab})
by higher order moment relaxation methods (see \cite[Chap.~6]{NiePoly}).

In this section, we show how to get an approximate optimizer pair
$(a^*, b^*)$ for \reff{pri-find-ab} when $G[y^*]$ is not separable.
Let $r = \rank\, G[y^*]$. We compute the spectral decomposition of $G[y^*]$ as
\be   \label{eq:spectral}
G[y^*]=\sum\limits_{i=1}^{r}d^{(i)}p^{(i)}(p^{(i)})^T ,
\ee
where each $d^{(i)}$ is a positive eigenvalue and
$p^{(i)}$ is the associated eigenvector.
Since $G[y^*] \succeq 0$, all $d^{(i)} > 0$.
The eigenvectors $p^{(i)}$ are orthonormal to each other.
For each $i=1\ddd r$, we select the vectors
\be  \label{eq:eig_decomp}
\boxed{
\begin{array}{lcl}
  \tilde{a}^{(i)} & \coloneqq & [p_1^{(i)},p^{(i)}_{1+n_2},\cdots ,p^{(i)}_{1+(n_1-1)n_2} ]^T,  \\
  \tilde{b}^{(i)} & \coloneqq & [ p^{(i)}_1,p^{(i)}_2,\cdots , p^{(i)}_{n_2} ]^T,  \\
  a^{(i)} & \coloneqq & \tilde{a}^{(i)}/\Vert \tilde{a}^{(i)}\Vert, \\
  b^{(i)} & \coloneqq & \tilde{b}^{(i)}/\Vert \tilde{b}^{(i)}\Vert.
\end{array} }
\ee
Each above pair $(a^{(i)},b^{(i)})$ is a feasible point for \reff{pri-find-ab}.
If  $(a^{(i)},b^{(i)})$ is nonsingular for \reff{eq-find-c},
we can solve a similar least squares problem and get the best vector $c^{(i)}$.
If it is singular, we can just let $c^{(i)} = 0$.
After this is done, we can select $(a^*,b^*, c^*)$ to be the best one among them, i.e.,
select the triple $(a^{(i)}, b^{(i)}, c^{(i)})$ with the minimum distance
\[
\| \mA - a^{(i)} \otimes b^{(i)} \otimes c^{(i)} \|_{\Omega}.
\]

\subsection{A convex relaxation algorithm}
Summarizing the above, we propose the following algorithm for
solving the biquadratic optimization \reff{pri-find-ab}.

\begin{alg}  \label{alg:moment}  \rm
For the partially given tensor $\mc{A}$, do the following:
\begin{itemize}

\item[Step~1] Solve the convex relaxation \reff{moment-for-ab} for a minimizer $y^*$.

\item[Step~2] If $\rank \, G[y^*] = 1$, compute $a^*,b^*$ as in \reff{eq:astarbstar},
then output $U = \{(a^*, b^*)\}$ and stop.
If $\rank \, G[y^*] > 1$, go to Step~3.

\item[Step~3] Apply Algorithm~4.2 in \cite{NieZhang16} to check if $G[y^*]$ is separable not.
If it is, compute the decomposition \reff{eq:sep_decomp},
then output $U \coloneq \{ (a^{(i)},b^{(i)}) : i=1\ddd r \}$.
If $G[y^*]$ is not separable, go to Step~4.

\item[Step~4] Compute the spectral decomposition \reff{eq:spectral} for $G[y^*]$.
Then, get vectors $a^{(i)},b^{(i)}$ from \reff{eq:eig_decomp} and output the set
$U \coloneq \{ (a^{(i)},b^{(i)}): i=1\ddd r \}$.

\end{itemize}	
\end{alg}

If the set $U$ is output in Step~2 or Step~3,
then each point in $U$ is a minimizer for \reff{pri-find-ab}.
If the set $U$ is output from Step~4,
then each point in $U$ is an approximate minimizer for \reff{pri-find-ab}.
We would like to remark that for some cases, the convex relaxation \reff{moment-for-ab}
may not be able to return a global minimizer for the quadratic optimization \reff{pri-find-ab}.
This can be seen in Example~\ref{exmp:choi}.
For such a case, we refer to the tight relaxation method in \cite[Chap.~6]{NiePoly}
for how to solve \reff{pri-find-ab}.

\section{Numerical experiments}
\label{sc:ne}

We present numerical experiments for solving rank-$1$
tensor completion with noises.
Note that for \reff{moment-for-ab}, the dimension of $y$ is
significantly larger than the length of $G[y]$ (see Table~\ref{tb:sdpnal+:r12}
for some typical cases).
The convex relaxation \reff{moment-for-ab} is solved
by the software \texttt{SDPNAL+} \cite{SunToh,Z-S-K}
with its default settings of parameters.
When $\rank\, G[y^*]=1$, the relaxation is tight 
and we apply (\ref{eq:astarbstar}) to get the global minimizer of (\ref{pri-find-ab}).
When $\rank\, G[y^*]>1$, we use {\tt GloptiPoly 3} \cite{GloPol3} 
and {\tt SeDuMi} \cite{sturm1999using} to check if it is separable or not.
The computations are implemented in MATLAB R2022b on a Lenovo
Laptop with CPU@2.10GHz and RAM 16.0G.
For a matrix in our numerical experiments,
its rank means the numerical rank, which is counted as the number of
singular values bigger than $10^{-6}$.
For neatness, all computed vectors/matrices are displayed with four decimal digits,
while the time consumption, errors and values of $\mA$
are displayed with two decimal digits.

For a partially given tensor $\mA \in \re^{n_1 \times n_2 \times n_3}$,
the density of observed entries is
\be\label{eq:den}
\texttt{den} = |\Omega| / (n_1 n_2 n_3) .
\ee
We call \texttt{den} the observation density for $\mA$.
For the rank-$1$ TCP (\ref{eq:tcp}), since we aim at finding rank-$1$ tensors close to the observed tensor $\mc{A}$ on the index set $\Omega$,
we measure the quality of a rank-$1$ completing tensor $a^* \otimes b^* \otimes c^*$ by the absolute and relative errors
\[
\texttt{err-abs} = \| \mA-(a^*\otimes b^*\otimes c^*)  \|_{\Omega}, \quad
\texttt{err-rel} = \frac{ \texttt{err-abs} }{ \| \mA \|_{\Omega} }.
\]
Moreover, we measure the difference between the computed completion $a^*\otimes b^*\otimes c^*$ and the noise-free (\texttt{nf}) rank-$1$ tensor $\widehat{\mA}$ (when it is available) as
\[
\texttt{err-nf}=\frac{\| \widehat{\mA}-(a^*\otimes b^*\otimes c^*)  \|}{\| \widehat{\mA} \|}.
\]

\subsection{Examples with $\rank\, G[y^*] = 1$}

\begin{Example}\rm %%%%%%%%%%%%
Consider the tensor $\mA \in \re^{ 3\times 4\times 3}$ with observed entries:
\[
\begin{array}{llll}
\mA_{111}=20.06, & \mA_{211}=40.02, &  \mA_{311}=20.15, &  \mA_{121}=-10.01,\\
\mA_{321}=-10.03, & \mA_{131}=40.35, & \mA_{341}=40.11, & \mA_{212}=40.08, \\
\mA_{312}=20.04,& \mA_{112}=20.14, & \mA_{122}=-10.09, & \mA_{322}=-10.03, \\
\mA_{222}=-20.12, & \mA_{142}=40.00, & \mA_{242}=80.57, & \mA_{113}=30.03, \\
\mA_{313}=30.18,& \mA_{133}=60.41, & \mA_{233}=121.13, & \mA_{333}=60.52. \\
\end{array}
\]
By Algorithm~\ref{alg:moment},  it took around 0.97 second
to solve (\ref{moment-for-ab}).
The largest eigenvalue of $G[y^{*}]$ is $\lambda_1 \approx 1.000$,
while the second largest one is around $10^{-12}$.
This implies that the numerical rank of $G[y^{*}]$ is one,
so the semidefinite relaxation (\ref{moment-for-ab}) is tight.
The minimizer $(a^*,b^*)$ of (\ref{pri-find-ab}) is
\[
\begin{array}{c}
 a^*  =  (0.4073,\,0.8169,\,0.4083), \quad
 b^*  =  (0.3276,\,-0.1642,\,0.6589,\,0.6579) .
\end{array}
\]
By Algorithm~\ref{alg:r1tc}, we get
\[
c^*  \,  =  \,  (149.9733,\,150.0009,\,225.0418).
\]
Finally, the errors for the completion $a^* \otimes b^* \otimes c^*$ are
\[
\begin{array}{c}
 \texttt{err-abs}  = 0.2996, \quad
 \texttt{err-rel}  =  0.0015, \quad
 \texttt{err-nf} = 0.0071. 
\end{array}
\]
\end{Example}

\begin{Example}\rm
Consider the tensor $\mA \in \re^{ 10\times 10\times 10}$ such that
\[
\mA_{ijk}  = \sin(i) \cos(j) \sin(k)  + \delta \mA_{ijk},
\]
where $\dt \mA_{ijk}$ is randomly generated by normal distribution
with zero mean and covariance $\sigma =10^{-4}$.
We select the index set
\[
\Omega  \coloneqq \{ (i,j,k) \in [10] \times [10] \times [10]:
i+j+k \equiv 0 \, \mod \, 3    \} .
\]
By Algorithm~\ref{alg:moment}, we get $y^*$ with $\rank\, G[y^*] =1$,
which took around $1.59$ seconds.
This can be checked by the eigenvalue decay of $G[y^{*}]$,
that the largest eigenvalue is $\lambda_1 \approx 1.000$, while the second largest one is around $10^{-12}$.
The numerical rank of $G[y^{*}]$ is one.
The semidefinite relaxation (\ref{moment-for-ab}) is tight,
and we get the minimizer $(a^*,b^*)$ of (\ref{pri-find-ab}):
\[
\begin{array}{ccl}
 a^*  &=& (-0.3763,\, -0.4066,\, -0.0631,\, 0.3384,\,0.4288,\,0.1249,\\
 \quad &\quad & -0.2938,\,-0.4424,\,-0.1843,\,0.2433),\\
 b^* & = &   (-0.2417,\, 0.1861,\, 0.4428,\, 0.2924,\,-0.1269,\,-0.4295,\\
 \quad &\quad & -0.3372,\,0.0651,\,0.4075,\,0.3753).
\end{array}
\]
By Algorithm~\ref{alg:r1tc}, we get
\[
\begin{array}{ccl}
c^*& =& (4.2076,\, 4.5467,\, 0.7056,\, -3.7842,\,-4.7949,\,-1.3971,\\
 \quad &\quad & 3.2851,\,4.9470,\,2.0607,\,-2.7202).\\
\end{array}
\]
The errors for the completion $a^* \otimes b^* \otimes c^*$ are
\[
\begin{array}{c}
 \texttt{err-abs}  = 1.79\cdot 10^{-4}, \quad
 \texttt{err-rel}  =  2.71\cdot 10^{-5},\quad
 \texttt{err-nf} = 5.31\cdot 10^{-5}.
\end{array}
\]
\end{Example}

\subsection{Examples with $G[y^*]$ separable}

\begin{Example}\rm
Consider the tensor $\mA \in \re^{ 3\times 3\times 4}$ with observed entries:
\[
\begin{matrix}
\mA_{311}=3.03 & \mA_{321}=6.02 &  \mA_{131}=3.02 &  \mA_{212}=1.52,\\
\mA_{112}=0.76 &\mA_{313}=0.76  &  \mA_{323}=1.53& \mA_{123}=0.51 ,\\
\mA_{233}=1.52& \mA_{333}=2.26  &  \mA_{134}=3.76.&
\end{matrix}
\]
By Algorithm~\ref{alg:moment}, it took around $0.99$ second to get the minimizer $y^*$ of (\ref{moment-for-ab}).
The biggest four eigenvalues of $G[y^{*}]$ are
\[
\lambda_1 \approx 0.4316, \, \lambda_2 \approx 0.3948, \,
\lambda_3 \approx 0.1736, \, \lambda_4 \approx 0.
\]
The numerical rank of $G[y^{*}]$ is three.
By applying the method in \cite{NieZhang16}, we certify that the matrix $G[y^*]$ is separable and get the decomposition \reff{eq:sep_decomp} with the vectors
\[
\begin{array}{ll}
a^{(1)} = (0.4471,0.8944,0.0000), & b^{(1)} = (0.9998,0.0002,0.0000) ;  \\
a^{(2)} = (0.0001,1.0000,0.0000), & b^{(2)} = (0.0002,0.9998,0.0000) ;  \\
a^{(3)} = (0.2693,0.5373,0.7992), & b^{(3)} = (0.2698,0.5385,0.7982) .
\end{array}
\]
This implies the tightness of the semidefinite relaxation (4.4), and the above pairs $(a^{(i)},b^{(i)})$ are minimizers for \reff{pri-find-ab}.
The $(a^{(1)},b^{(1)})$ and $(a^{(2)},b^{(2)})$ are singular for \reff{eq-find-c},
but $(a^{(3)},b^{(3)})$ is nonsingular. By Algorithm~\ref{alg:r1tc}, we get
\[
c^{(3)} = (14.0396,\, 10.4821,\,3.5440,\,17.4921).
\]
The errors for the completion $a^* \otimes b^* \otimes c^*$ are
\[
\begin{array}{c}
 \texttt{err-abs}  = 0.0091, \quad
 \texttt{err-rel}  =  0.0010, \quad
 \texttt{err-nf} = 0.0077.
\end{array}
\]
\end{Example}

\begin{Example}  \rm
Consider the rank-$1$ tensor $\mA \in \re^{ 3\times 3\times 3}$ with observed entries:
\[
\begin{array}{lllll}
\mA_{131}=4, & \mA_{133}=4, &  \mA_{213}=1, &  \mA_{112}=4 , & \mA_{232}=16,\\
\mA_{122}=4, & \mA_{312}=2, &  \mA_{211}=1, &  \mA_{322}=2 , & \mA_{333}=2,\\
\mA_{331}=2, & \mA_{221}=1, &  \mA_{223}=1. &    & \\
\end{array}
\]
By Algorithm~\ref{alg:moment}, it took around $0.81$ second to get the minimizer $y^*$ of (\ref{moment-for-ab}).
The matrix $G[y^*]$ is separable. The biggest three eigenvalues of $G[y^{*}]$ are
\[
\lambda_1 \approx 0.5431, \, \lambda_2 \approx 0.4569, \, \lambda_3\approx 0.
\]
The numerical rank of $G[y^{*}]$ is two.
By applying the method in \cite{NieZhang16}, we detect that
the matrix $G[y^*]$ is separable
and get the decomposition \reff{eq:sep_decomp} with the vectors
\[
\begin{array}{ll}
a^{(1)} = (\frac{2}{3},-\frac{2}{3},\frac{1}{3}), & b^{(1)} = (-\frac{\sqrt{2}}{6},-\frac{\sqrt{2}}{6},\frac{2\sqrt{2}}{3}) ,  \\
a^{(2)} = (\frac{2}{3},\frac{2}{3},\frac{1}{3}), & b^{(2)} = (\frac{\sqrt{2}}{6},\frac{\sqrt{2}}{6},\frac{2\sqrt{2}}{3}).
\end{array}
\]
Both $(a^{(1)},\,b^{(1)})$ and $(a^{(2)},\,b^{(2)})$ are nonsingular.
By Algorithm~\ref{alg:r1tc}, we get
\[
\begin{array}{ll}
c^{(1)} = (\frac{9\sqrt{2}}{2},\,-18\sqrt{2},\,\frac{9\sqrt{2}}{2}), &
c^{(2)} = (\frac{9\sqrt{2}}{2},\,18\sqrt{2},\,\frac{9\sqrt{2}}{2}).
\end{array}
\]
They give two exact rank-$1$ completions $a^{(i)}\otimes b^{(i)}\otimes c^{(i)}$,
whose completion errors are all zeros.
\end{Example}

\subsection{Examples for $G[y^*]$  not separable}

We remark that the convex relaxation \eqref{moment-for-ab}
may not be tight for solving the biquadratic optimization \reff{pri-find-ab}, i.e.,
the optimal value $p^*$ of (\ref{moment-for-ab})
may be smaller than the minimum value $f^*$ of \reff{pri-find-ab}.
This happens in the following example.

\begin{Example} \label{exmp:choi} \rm
Consider the tensor $\mA \in \re^{ 3\times 3\times 7}$
with the partially given entries:
\[
\begin{array}{c}
\mA_{111} = \mA_{221} = \mA_{331} = \mA_{115} = \mA_{116} = \mA_{117} = 1,\\
\mA_{122} = \mA_{233} = \mA_{314} = \mA_{135} = \mA_{216} = \mA_{327} = 0,\\
\mA_{112} = \mA_{113} = \mA_{114} = \sqrt{3}.
\end{array}
\]
One can check that
$f(a,b) = \Vert a\otimes b \Vert^2 + h(a,b),$
where $h(a,b)$ is the Choi's biquadratic form \cite{Choi75}:
\[
\begin{array}{l}
h(a,b) \ = \ a_1^2b_1^2 + a_2^2b_2^2 + a_3^2b_3^2 + 2(a_1^2b_2^2+a_2^2b_3^2+a_3^2b_1^2 )\\
\qquad\qquad\qquad \qquad\qquad - 2a_1a_2b_1b_2 - 2a_1a_3b_1b_3 - 2a_2a_3b_2b_3.
\end{array}
\]
It is well-known that $h(a,b)$ is nonnegative everywhere but is not a sum-of-squares.
The minimum value of \reff{pri-find-ab} is $f^* = 1$, achieved at $a=b=(1,0,0)$.
However, since $h$ is not a sum-of-squares, it holds
(recall that $\mc{Q}$ is the cone of sums of squares of bilinear forms in $(a,b)$; see (\ref{sos-for-ab}))
\[ f-f^*\Vert a\otimes b \Vert^2 = h\notin  \mc{Q}. \]
This implies that $d^* = p^* < f^*$, since the optimal value $d^*$ of (\ref{sos-for-ab})
is attainable by Proposition~\ref{prop:relx}.
Indeed, by solving (\ref{moment-for-ab}) numerically, we get
\[
p^* = d^* \approx  0.9028 < f^* = 1.
\]
By Algorithm~\ref{alg:moment}, we get $y^*$ with the matrix $G[y^*]$ computed as
\begin{scriptsize}
\[
\left[
\begin{array}{rrrrrrrrr}
0.0569 & 0.0000 &0.0000& -0.0000 & 0.0569 & 0.0000 & -0.0000 & -0.0000 & 0.0569\\
-0.0000 & 0.0122 & -0.0000 & 0.0569 &0.0000 & 0.0000 &-0.0000 & 0.0000 & 0.0000\\
0.0000 & -0.0000 & 0.2642 & 0.0000 & 0.0000 & -0.0000& 0.0569 & -0.0000 & 0.0000\\
-0.0000 & 0.0569& 0.0000 & 0.2642 & -0.0000 & 0.0000 & -0.0000 & -0.0000 & -0.0000\\
0.0569 & 0.0000 & 0.0000 & -0.0000 & 0.0569 & -0.0000 & -0.0000 & -0.0000 & 0.0569\\
0.0000& 0.0000& -0.0000 & 0.0000 & -0.0000& 0.0122 & -0.0000 & 0.0569 & 0.0000\\
-0.0000 & -0.0000 & 0.0569 & -0.0000 & -0.0000 & -0.0000 & 0.0122 & 0.0000 & -0.0000\\
-0.0000 & 0.0000 & -0.0000 & -0.0000& -0.0000 & 0.0569& 0.0000 & 0.2642 & -0.0000\\
0.0569& -0.0000 & 0.0000& -0.0000 & 0.0569 & 0.0000 & -0.0000& -0.0000 & 0.0569\\
\end{array} \right].
\]
\end{scriptsize}
It is not separable, detected by the method in \cite{NieZhang16}.
We have $\rank\, G[y^*] = 4$.
The four positive eigenvalues are
\[
d^{(1)}=d^{(2)}=d^{(3)}=0.2765,\quad d^{(4)}= 0.1706.
\]
The associated eigenvectors are respectively:
\[
\begin{array}{l}
  p^{(1)}  =(0.0000,\,0.0283,\,0.0222,\,0.1314,\,0.0000,\,-0.2085,\,0.0048,\,-0.9685,\,0.0000),\\
  p^{(2)}  =(0.0000,\,0.2085,\,0.0026,\,0.9687,\,0.0000,\,0.0283,\,0.0006,\,0.1315,\,0.0000),\\
  p^{(3)}  =(0.0000,\,-0.0012,\,0.9774,\,-0.0056,\,0.0000,\,0.0047,\,0.2104,\,0.0216,\,0.0000),\\
  p^{(4)}  =(0.5774,\,0.0000,\,0.0000,\,0.0000,\,0.5774,\,0.0000,\,-0.0000,\,0.0000,\,0.5774).\\
\end{array}
\]
By (\ref{eq:eig_decomp}), we get:
\[ \begin{array}{ll}
	a^{(1)} =(0.0000,\,0.9993,\,0.0363), & b^{(1)} = (0.0000\,0.7868,\,0.6172),\\
	a^{(2)} = (0.0000,\,1.0000,\,-0.0006),& b^{(2)} = (0.0000\,0.9999,\,0.0126),\\
	a^{(3)} =(0.0000,\,0.0266,\,0.9996), & b^{(3)} = (0.0000,\,-0.0012,\,1.0000),\\
	a^{(4)} = (1.0000,\,0.0000,\,-0.0000),& b^{(4)} = (1.0000,\,0.0000,\,0.0000).\\
\end{array}
\]
The $(a^{(1)},b^{(1)})$, $(a^{(2)},b^{(2)})$ and $(a^{(3)},b^{(3)})$ are singular for \reff{eq-find-c},
but $(a^{(4)},b^{(4)})$ is nonsingular. By Algorithm~\ref{alg:r1tc}, we get
\[
c^{(4)} = (14.0396,\, 10.4821,\,3.5440,\,17.4921).
\]
The errors for the completion $a^{(4)} \otimes b^{(4)} \otimes c^{(4)}$ are
\[
\begin{array}{c}
 \texttt{err-abs}  = 1.4142, \quad
 \texttt{err-rel}  =  0.3651,\quad
 \texttt{err-nf} = 0.3651.
\end{array}
\]

\end{Example}

\subsection{The performance for some random instances}

We first explore the distribution of $\rank\, G[y^*]$ for the optimizer $y^*$
of the convex relaxation \eqref{moment-for-ab} for some random problems.

\begin{Example}
\label{ex:rank_vs_den}  \rm
Let $n_1=n_2=n_3 =n$. For given $n$, the density \texttt{den} and a small positive scalar $\sig$,
we generate $\mA \in\re^{n \times n\times n}$ as
\[
\mA = \hat{a}\otimes \hat{b}\otimes \hat{c} + \dt\mA,
\]
where the $\hat{a}, \hat{b}, \hat{c}$ are randomly generated
by the {\tt MATLAB} function {\tt randn(n,1)}.
Each entry of $\dt\mA$ is generated randomly obeying the normal distribution
with zero mean and variance $\sig^2$.
The index set
$\Omega$ is randomly chosen using {\tt sprand} such that
$|\Omega| = \lceil n^3\cdot  \texttt{den} \rceil$.
For each given triple $(n,\, \texttt{den},\, \sig)$, we randomly generate $20$ random instances of
$\mA$ and the index set $\Omega$.
We record the percentage of instances with
$\rank\, G[y^*] =1$, $\rank\, G[y^*] =2$ and $\rank\, G[y^*] \ge 3$.
The distributions of ranks are shown in Table~\ref{tb:rank:percen}.
It is interesting to observe that $\rank\,G[y^*]$ is more likely to be $1$
as the density \texttt{den} increases.
A possible reason for this phenomena is that
the underlying rank-$1$ tensor $\hat{a} \otimes \hat{b} \otimes \hat{c}$
may be unique when the density is high, so that $G[y^*]$ has rank one.
\begin{table}[htbp]
\caption{The distribution of $\rank\, G[y^*]$ for some random TCPs.}
\label{tb:rank:percen}
\scalebox{0.85}{
\begin{tabular}{cccc|cccc}
\specialrule{.2em}{0em}{0.1em}
 $(n, \texttt{den}, \sig)$   &  $rank=1$  & rank=2 & rank$\ge 3$ & $(n, \texttt{den}, \sig)$   &  rank=1  & rank=2 & rank$\ge 3$ \\
$(11, 0.12, 10^{-4})$  & 10\%   &20\%   & 70\%
& $(16, 0.11, 10^{-2})$  &20\%  & 40\%  & 40\% \\
$(11, 0.18, 10^{-2})$ & 55\%   & 40\% & 5\%
&$(16, 0.12, 10^{-3})$  & 70\%   &25\%   & 5\%    \\
$(11, 0.23, 10^{-3})$  & 100\%  & 0  & 0
&$(16, 0.16, 10^{-4})$ & 100\% & 0  & 0\\
\specialrule{.2em}{0em}{0.1em}
$(12, 0.13, 10^{-4})$  & 20\%   & 50\%   & 30\%
& $(17, 0.10, 10^{-2})$  &10\%   & 40\%   & 50\%      \\
$(12, 0.19, 10^{-2})$  & 15\%   & 85\% & 0
& $(17, 0.12, 10^{-4})$  & 70\% & 20\%  & 10\%     \\
$(12, 0.25, 10^{-3})$  & 100\%  & 0  & 0
& $(17, 0.17, 10^{-3})$  & 100\% & 0  & 0   \\
\specialrule{.1em}{.1em}{0.1em}
 $(13, 0.10, 10^{-4})$ & 5\%   & 10\%   & 85\%
 & $(18, 0.10, 10^{-2})$  &10\%   & 60\%   & 30\%  \\
  $(13, 0.15, 10^{-2})$  &60\%   & 40\% & 0
&  $(18, 0.15, 10^{-4})$  & 90\% & 10\%  & 0 \\
 $(13, 0.21, 10^{-3})$  & 100\%  & 0  & 0
 &  $(18, 0.17, 10^{-3})$   & 100\% & 0  & 0 \\
\specialrule{.1em}{.1em}{0.1em}
 $(14, 0.10, 10^{-2})$  & 5\%   & 20\%   & 75\%
& $(19, 0.09, 10^{-2})$   &10\%   & 40\%   & 50\%     \\
 $(14, 0.15, 10^{-4})$  & 80\%   & 20\%  & 0
& $(19, 0.10, 10^{-3})$  & 20\%    & 60\%   & 20\%   \\
 $(14, 0.20, 10^{-3})$  & 100\% &  0 & 0
 & $(19, 0.15, 10^{-4})$  & 100\% & 0  & 0   \\
\specialrule{.1em}{.1em}{0.1em}
 $(15, 0.11, 10^{-2})$  &20\%   & 50\%   & 30\%
& $(20, 0.09, 10^{-3})$  &10\%   & 20\%   & 70\%   \\
 $(15, 0.15, 10^{-3})$ & 80\%  & 10\%  & 10\%
&   $(20, 0.10, 10^{-4})$  & 40\%    & 10\%   &50\%\\
 $(15, 0.20, 10^{-3})$  & 100\%   & 0 & 0
 &  $(20, 0.15, 10^{-3})$  & 100\% & 0  & 0   \\
\specialrule{.1em}{.1em}{0.1em}
\end{tabular}}
\end{table}
\end{Example}

\begin{Example}
\label{exm:errors:rand}  \rm
We explore the performance of the convex relaxation \eqref{moment-for-ab}.
For every $(n,\,den,\, \sig)$, we randomly generate $20$ random instances of noisy tensor $\mA$
and select the index set $\Omega$ as in Example~\ref{ex:rank_vs_den}.
For each instance, we solve the convex relaxation \eqref{moment-for-ab}.
Since the level of the noise varies, we consider the ratio
\[
\texttt{err-rat}=\frac{\| \mA-(a^*\otimes b^*\otimes c^*)  \|_{\Omega} }{\| \delta\mA \|_{\Omega} },
\]
so that the performance of the algorithm under different noise levels can be compared fairly.
The numerical performance is presented in Table~\ref{tb:sdpnal+:r12}.
We report the minimum and maximum values of {\tt err-rat} and {\tt err-nf},
and the average computational time (in seconds)
over all $20$ instances for each $(n,\,den,\, \sigma)$.
The dimension of $y$ and the length of the matrix $G[y]$
in \eqref{moment-for-ab} are shown in columns
labeled by \texttt{dim\,}$y$ and \texttt{len\,}$G[y]$ respectively.

\begin{table}[htbp]
\centering
\caption{The performance of the convex relaxation \eqref{moment-for-ab}.}
\label{tb:sdpnal+:r12}
\scalebox{0.82}{
\begin{tabular}{cccccc cc c cc}
\specialrule{.2em}{0em}{0.1em}
\multirow{2}{*}{$den$} & \multirow{2}{*}{$(n_1,n_2,n_3)$}& \multirow{2}{*}{$\sigma$}&  \multirow{2}{*}{\texttt{dim\,}$y$} & \multirow{2}{*}{\texttt{len\,}$G[y]$} & \multirow{2}{*}{\texttt{time}}
&\multicolumn{2}{c}{ \texttt{err-rat} } &&\multicolumn{2}{c}{ \texttt{err-nf} } \\
\cmidrule{7-8}\cmidrule{10-11}
& & &   && &  min  &  max & & min  &  max  \\
\specialrule{.1em}{.1em}{0.1em}
0.25  & (10,\,10,\,10) & $10^{-3}$ & 3025 & 100 & 0.87  &  0.7841  & 0.9372  && $2.09\cdot 10^{-3}$ & $3.99\cdot 10^{-3}$   \\
\specialrule{.1em}{.1em}{0.1em}
0.35 & (11,\,12,\,10) & $10^{-2}$ & 5148 & 132 & 1.27    & 0.8058   & 0.9697 && $1.31\cdot 10^{-3}$ & $3.11\cdot 10^{-3}$  \\
\specialrule{.1em}{.1em}{0.1em}
0.20  & (12,\,13,\,14) & $10^{-3}$ & 7098 & 156 &  1.73   & 0.7702 & 0.9301 &&$1.51\cdot 10^{-3}$ & $3.24\cdot 10^{-3}$       \\
\specialrule{.1em}{.1em}{0.1em}
0.16 & (13,\,14,\,15) & $10^{-4}$ & 9555 & 182 &  2.08   & 0.7417 & 0.9330 &&$1.86\cdot 10^{-4}$& $3.99\cdot 10^{-4}$ \\
\specialrule{.1em}{.1em}{0.1em}
0.18  & (15,\,15,\,16) & $10^{-3}$ & 14400 & 225&  3.32   & 0.8009  & 0.9586 && $1.82\cdot 10^{-3}$ & $4.11\cdot 10^{-3}$  \\
\specialrule{.1em}{.1em}{0.1em}
0.17  & (16,\,18,\,18) & $10^{-2}$ & 23256 & 288 & 3.41  & 0.8478  &   0.9653 &&$1.58\cdot 10^{-3}$ & $3.48\cdot 10^{-3}$  \\
\specialrule{.1em}{.1em}{0.1em}
0.15 & (19,\,20,\,18) & $10^{-3}$ &39900& 380 & 11.18  & 0.8820   &   0.9587    && $1.36\cdot 10^{-2}$ & $1.58\cdot 10^{-2}$  \\
\specialrule{.1em}{.1em}{0.1em}
0.12  & (20,\,20,\,20) & $10^{-3}$ & 44100& 400 & 6.59  &0.8826  &   0.9612  &&$1.52\cdot 10^{-3}$& $3.48\cdot 10^{-3}$\\
\specialrule{.1em}{.1em}{0.1em}
0.12  & (21,\,22,\,22) & $10^{-4}$ & 58443 & 462 & 7.66 & 0.8534 &  0.9494  &&$2.55\cdot 10^{-3}$& $3.08\cdot 10^{-3}$\\
\specialrule{.1em}{.1em}{0.1em}
0.12  & (23,\,24,\,25) & $10^{-3}$ & 82800& 552 & 9.63 & 0.9016 &  0.9697 &&$2.65\cdot 10^{-4}$&  $4.08\cdot 10^{-4}$ \\
\specialrule{.1em}{.1em}{0.1em}
0.12  & (25,\,25,\,25) & $10^{-4}$ & 105625 & 625 & 13.10  &0.9246 & 0.9526 && $2.31\cdot 10^{-3}$ & $3.17\cdot 10^{-3}$  \\
 \specialrule{.2em}{0em}{0.1em}
\end{tabular}}
\end{table}
\end{Example}

\subsection{A comparison between two optimization models}

A traditional approach for computing rank-$1$ tensor completion is
to solve the nonlinear optimization problem \eqref{eq-least-sec1},
while we solve the optimization problem \reff{eq-tcn-1}.
The problem \eqref{eq-least-sec1} can be solved
by nonlinear least squares (NLS) methods,
whose performance highly depends on the choice of starting points.
In practice, people often choose random ones.
In contrast, the problem \reff{eq-tcn-1}
can be solved by the convex relaxation \eqref{moment-for-ab},
which does not depend on starting points.
The following is a comparison between \reff{eq-least-sec1} and \reff{eq-tcn-1}.

\begin{Example}\rm \label{ex-nls}
Consider the TCP with $n_1 = n_2 = n_3= n$.
We compare our method with solving \eqref{eq-least-sec1} by Levenberg–Marquardt method,
which is implemented in the MATLAB function {\tt lsqnonlin},
with the default settings and a randomly generated initial point.
For each $(n,\, \texttt{den},\,\sig)$, we randomly generate $10$
instances as in Example~\ref{ex:rank_vs_den}.
We compare the minimum and maximum value of \texttt{err-nf} for both methods.
The numerical results are shown in Table~\ref{tb:comp:nls}.
In the table, \texttt{dim\,}$y$
stands for the dimension of the vector $y$ in (\ref{moment-for-ab}).
The relative errors for {\tt lsqnonlin} to solve \eqref{eq-least-sec1}
are shown in the column labeled by \texttt{err-nls},
while those errors for solving \reff{eq-tcn-1}
by the semidefinite relaxation \eqref{moment-for-ab}
are shown in the column labeled by \texttt{err-sdp}.
\begin{table}[htb]
\centering
\caption{Comparison between \reff{eq-least-sec1} and \reff{eq-tcn-1}. }
\scalebox{0.9}{
\btab{ccccccrccrr}  \specialrule{.2em}{0em}{0.1em}
\multirow{2}{*}{$den$}   & \multirow{2}{*}{$(n_1,n_2,n_3)$}   & \multirow{2}{*}{$\sig$}&\multirow{2}{*}{\texttt{dim\,}$y$} &
 \multicolumn{2}{c}{ \texttt{err-nls} } &  & \multicolumn{2}{c}{\texttt{err-sdp} }   \\
\cmidrule{5-6}  \cmidrule{8-9}
  &    & & & min  &  max   && min  & max   \\
\specialrule{.1em}{.1em}{0.1em}
0.50 & (10,10,10) & $10^{-3}$ &3025   & $1.71\cdot 10^{-3}$  & $8.74\cdot 10^{3}$
&& $1.90\cdot 10^{-3}$ &$4.90\cdot 10^{-3}$   \\
\specialrule{.1em}{.1em}{0.1em}
0.50 & (11,12,10) & $10^{-3}$ &5148  & $1.96\cdot 10^{-3}$  & $9.23\cdot 10^{3}$
&&$1.97\cdot 10^{-3}$ &$3.80\cdot 10^{-3}$    \\
\specialrule{.1em}{.1em}{0.1em}
0.30 & (12,12,13) & $10^{-4}$ &6084 & $2.52\cdot 10^{-4}$ & $1.71\cdot 10^{4}$
&&$2.64\cdot 10^{-3}$ &$9.62\cdot 10^{-2}$   \\
\specialrule{.1em}{.1em}{0.1em}
0.35 & (13,14,13) & $10^{-3}$&9555  &$1.85\cdot 10^{-3}$  &$4.01\cdot 10^{3}$
&&    $1.85\cdot 10^{-3}$ &$3.60\cdot 10^{-3}$  \\
\specialrule{.1em}{.1em}{0.1em}
0.25 & (15,14,15) & $10^{-2}$ &12600& $1.65\cdot 10^{-3}$ &$9.16\cdot 10^{5}$
&& $1.67\cdot 10^{-2}$  & $4.21\cdot 10^{-2}$  \\
\specialrule{.1em}{.1em}{0.1em}
0.30 & (15,10,10) & $10^{-3}$  &6600 & $1.59\cdot 10^{-3}$  &  $3.38\cdot 10^{3}$
&&  $1.54\cdot 10^{-3}$  & $3.47\cdot 10^{-3}$\\
\specialrule{.1em}{.1em}{0.1em}
0.22 & (15,13,15) & $10^{-2}$ &10920 & $1.76\cdot 10^{-2}$ &$3.41\cdot 10^{3}$
&&$1.74\cdot 10^{-2}$ & $3.69\cdot 10^{-2}$  \\
\specialrule{.1em}{.1em}{0.1em}
0.15 & (15,15,15) & $10^{-4}$&14400& $2.16\cdot 10^{-4}$&  $7.30\cdot 10^{3}$
&& $2.08\cdot 10^{-4}$ & $4.30\cdot 10^{-4}$ \\
\specialrule{.1em}{.1em}{0.1em}
0.15 & (16,17,18) & $10^{-3}$ &20808 &$2.11\cdot 10^{-3}$ & $3.43\cdot 10^{6}$
&& $2.14\cdot 10^{-3}$  & $3.22\cdot 10^{-3}$  \\
\specialrule{.1em}{.1em}{0.1em}
0.15 & (17,18,19) & $10^{-2}$&26163 &$2.57\cdot 10^{-4}$ & $3.07\cdot 10^{6}$
& &$2.07\cdot 10^{-4}$  & $3.60\cdot 10^{-4}$\\
\specialrule{.1em}{.1em}{0.1em}
0.12 & (20,20,20) & $10^{-4}$&44100 & $3.23\cdot 10^{-4}$ & $3.20\cdot 10^{3}$
&& $3.37\cdot 10^{-4}$ &$8.39\cdot 10^{-2}$\\
\specialrule{.2em}{0em}{0.1em}
\etab
}
\label{tb:comp:nls}
\end{table}
\end{Example}

\begin{Example}\rm
Consider TCPs with $n = n_1=n_2=n_3$. We compare the performance of
SDP solvers {\tt Mosek} \cite{Mosek}
and {\tt SDPNAL+} for solving $\eqref{moment-for-ab}$.
The default settings of {\tt Mosek} are exploited. For each $(n,\, \texttt{den},\,\sig)$, we randomly generate $10$ instances as in Example~\ref{ex:rank_vs_den}.
We report the average computational time (in seconds) for these $10$ instances.
The time consumed by {\tt Mosek} is denoted as $\tt t_{mosek}$,
and that time of {\tt SDPNAL+} is denoted as $\tt t_{sdpnal+}$.
We also compare the minimum and maximum {\texttt{err-nf}} for both SDP solvers,
which are denoted by \texttt{err-mosek} and \texttt{err-sdpnal+} respectively.
The computational results are reported in Table~\ref{ta:com:mose:sdp}.
As we can see, the errors of computed tensor completions using
{\tt Mosek} and {\tt SDPNAL+} are similar,
but {\tt SDPNAL+} has clear advantages in computational time
when the dimensions are large (e.g., $n> 10$).
Moreover, the computer was out of memory (oom) when
{\tt Mosek} is applied to solve (\ref{moment-for-ab}) when $n\ge 18$.
%%%%%%%%%%%%%%%%%%%%%%%%%%%%%%%%%%%%%%%%%%%%%%%%%
\begin{table}[htbp]
\caption{Compare {\tt Mosek} with {\tt SDPNAL+}  for solving $\eqref{moment-for-ab}$ }
\label{ta:com:mose:sdp}
\scalebox{0.76}{
\begin{tabular}{cccccccccccc}  \specialrule{.2em}{0em}{0.1em}
\multirow{2}{*}{$den$}   & \multirow{2}{*}{$n$} &\multirow{2}{*}{$\delta\mA$} & \multirow{2}{*}{y} & \multirow{2}{*}{$G[y]$} &\multirow{2}{*}{$\tt t_{mosek}(s)$} &
\multirow{2}{*}{$\tt t_{sdpnal+}(s)$} & \multicolumn{2}{c}{ \texttt{err-mosek} } & &  \multicolumn{2}{c}{\texttt{err-sdpnal+} }   \\
\cmidrule{8-9} \cmidrule{11-12}
& & &  & & & & min  &  max   && min  & max   \\
\specialrule{.1em}{.1em}{0.1em}
0.46 & 8& $10^{-3}$ &1296  &64 & 0.16 &  0.54 & $4.07\cdot 10^{-4}$ & $6.45\cdot 10^{-4}$& & $4.07\cdot 10^{-4}$ & $6.45\cdot 10^{-4}$  \\
\specialrule{.1em}{.1em}{0.1em}
0.45 & 9& $10^{-3}$  &2025&81 &0.43  & 0.69 & $3.48\cdot 10^{-4}$ & $7.33\cdot 10^{-4}$& & $3.48\cdot 10^{-4}$ & $7.33\cdot 10^{-4}$ \\
\specialrule{.1em}{.1em}{0.1em}
0.40 & 10& $10^{-2}$&3025   & 100& 0.96 &  0.81 & $2.40\cdot 10^{-3}$ & $5.41\cdot 10^{-3}$& & $2.40\cdot 10^{-3}$ & $5.41\cdot 10^{-3}$ \\
\specialrule{.1em}{.1em}{0.1em}
0.35 & 11& $10^{-4}$  &4356  & 121 & 2.75 & 0.99 & $4.03\cdot 10^{-5}$ & $4.83\cdot 10^{-5}$& & $4.00\cdot 10^{-5}$ & $4.83\cdot 10^{-3}$ \\
\specialrule{.1em}{.1em}{0.1em}
0.30 & 12& $10^{-3}$&6084  & 144 & 6.64  &  1.30 & $3.42\cdot 10^{-4}$ & $6.97\cdot 10^{-4}$& & $3.42\cdot 10^{-4}$ & $6.97\cdot 10^{-4}$\\
\specialrule{.1em}{.1em}{0.1em}
0.25 & 13& $10^{-2}$  & 8281& 169& 14.72  & 1.92 & $3.12\cdot 10^{-3}$ & $5.97\cdot 10^{-3}$& & $3.12\cdot 10^{-3}$ & $5.97\cdot 10^{-3}$\\
\specialrule{.1em}{.1em}{0.1em}
0.23 & 14& $10^{-3}$  &11025& 196& 36.11  & 2.98 & $3.78\cdot 10^{-4}$ & $7.06\cdot 10^{-4}$& & $3.78\cdot 10^{-4}$ & $7.06\cdot 10^{-4}$ \\
\specialrule{.1em}{.1em}{0.1em}
0.22 & 15& $10^{-4}$  &14400& 225& 75.58  & 3.37 & $2.74\cdot 10^{-5}$ & $7.11\cdot 10^{-5}$& & $2.72\cdot 10^{-5}$ & $7.11\cdot 10^{-4}$ \\
\specialrule{.1em}{.1em}{0.1em}
0.20 & 16& $10^{-3}$  &18496& 256& 170.13 & 4.97 & $3.41\cdot 10^{-4}$ & $3.90\cdot 10^{-4}$& & $3.41\cdot 10^{-4}$ & $3.91\cdot 10^{-4}$ \\
\specialrule{.1em}{.1em}{0.1em}
0.18 & 17& $10^{-3}$  &23409& 289& 364.17 & 9.75 &$3.66\cdot 10^{-4}$ & $4.23\cdot 10^{-4}$& & $3.66\cdot 10^{-4}$ & $4.23\cdot 10^{-3}$ \\
\specialrule{.1em}{.1em}{0.1em}
0.19 & 18& $10^{-2}$  &29241& 324& oom & 11.07& oom & oom & & $3.21\cdot 10^{-3}$ & $4.51\cdot 10^{-3}$ \\
\specialrule{.1em}{.1em}{0.1em}
0.17 & 19& $10^{-3}$  & 36100 & 361& oom & 13.63 & oom & oom & & $3.21\cdot 10^{-3}$ & $4.51\cdot 10^{-3}$ \\
\specialrule{.1em}{.1em}{0.1em}
0.16 & 20 & $10^{-4}$ & 44100 & 400 & oom & 18.36 & oom & oom & & $4.05\cdot 10^{-4}$ & $6.42\cdot 10^{-4}$ \\
\specialrule{.2em}{0em}{0.1em}
\end{tabular}}
\end{table}
\end{Example}

\subsection{Applications in image recovery and low rank tensor completion}
\label{sc:application}

Tensor completion has important applications in image recovery.
A color image can be represented as an $n_1\times n_2\times n_3$ tensor $\mc{A}$,
where $n_1 \times n_2$ is the resolution and $n_3=3$ represents three channels of colors (red, green, and blue). Typically, the tensor $\mc{A}$ has noises and missing entries, due to physical sensor defects, errors in data transmission, or corruption during storage.
On the other hand, image tensors are usually low rank since the pixels exhibit strong local smoothness and the color channels are highly correlated.
Thus, for the partially given noisy tensor $\mc{A}$ with the index set of observations $\Omega$, one is usually interested in finding a low rank tensor $\widehat{\mA}$ such that
\[ (\mc{A} - \widehat{\mA})_{ijk} \quad
\mbox{is small for} \quad (i,j,k)\in \Omega.\]

Algorithm~\ref{alg:r1tc} mainly deals with rank-$1$ tensor completions with noises.
It can also be generalized to find low-rank tensor completions.
This is particularly useful in recovering real-world images.
% We extend our approach to find low-rank completions for image tensors.
As shown in the recent work \cite{Cifuentes2025},
one can obtain low rank tensor completions
by repeatedly computing rank-$1$ completions.
We describe how to do this in the following.

For a partially given tensor $\mc{A}\in \re^{n_1\times n_2\times n_3}$
with the index set $\Omega$ and a priori given rank $r$,
let $\mA^* =0$ be the all-zero tensor of the same dimension.
Initialize $\ell\coloneqq 0$, $\mA^{(0)}:= \mA$.
For $\ell < r$, do the following:
\begin{itemize}

\item[1.] Apply Algorithm~\ref{alg:r1tc} to find a rank-$1$ completion $a^{(\ell)}\otimes b^{(\ell)}\otimes c^{(\ell)}$ for $\mA^{(\ell)}$.

\item[2.] If $\widetilde{A}\coloneqq \mA^* + a^{(\ell)}\otimes b^{(\ell)}\otimes c^{(\ell)}$ gives a better completion (i.e., the error of $\widetilde{A}$ is smaller than that of $\mA^*$), then update
    \[
    \mA^* \coloneqq \widetilde{A}, \quad
    \mA^{(\ell + 1)} \coloneqq  \mA^{(\ell)} - a^{(\ell)}\otimes b^{(\ell)}\otimes c^{(\ell)},
    \]
and let $\ell\coloneqq \ell + 1$.
Otherwise, output $\mA^*$ as the low rank completion for $\mA$ and stop.

\end{itemize}

We apply the above procedure to complete real-world image tensors
$\mc{A}$ with $n_1 = 128$, $n_2 = 96$.
We consider the cases where the density {\tt den} of observed entries,
given as in (\ref{eq:den}), is $0.5$ and $0.7$.
Since the PSD variable in the semidefinite program (\ref{moment-for-ab}) is $n_1n_2$-by-$n_1n_2$,
the computer is out of memory to implement Algorithm~\ref{alg:r1tc} directly.
Instead, we partition this image tensor into $8\times 6$ smaller tensors,
each of size $16\times 16\times 3$, and apply our method on each one of them.
Moreover, since the total number of pixels $n_1n_2$ is relatively large,
there may exist a pair $(i,j)$ such that $(i,j,k)\notin \Omega$ for all $k$.
For such a pair $(i,j)$, we heuristically let
\[\begin{gathered}
\Theta \coloneqq \{ (i',j') : i-1\le i'\le i+1, j-1\le j'\le j+1  \},\\
\mc{A}_{ijk} \coloneqq \frac{1}{|\Omega_k\cap \Theta|}
\sum_{(i',j')\in \Omega_k\cap \Theta} \mc{A}_{i'j'k}.
\end{gathered}\]
Then, we add the above triples $(i,j,k)$ to $\Omega$ before doing tensor completions.
The recovered images are shown in Figure~\ref{fig:img_real}.
For ${\tt den} = 0.5$ and $0.7$,
the column ``observation'' shows the image corresponds to the partially given tensor $\mc{A}$,
where the missing entries are highlighted in distinct colors.
For $r = 1,2,3$, the images under ``rank $\le r$'' are produced
by the above procedure with the maximal rank equal to $r$.
Finally, we remark that in our computation,
the minimizer matrices of (\ref{moment-for-ab}) are all rank-$1$ when $\ell = 1$,
and are mostly rank-$1$ when $\ell > 1$.
For instance, there are $2$ exceptions for $\ell=2$ and one exception for
$\ell =3$ among $48$ blocks when ${\tt den} = 0.5$.
All exceptions occur when the error of the last iterate is small (e.g., ${\tt err-rel}< 0.0015$),
which means the remainder $\mc{A}^{(\ell-1)}$
is small and computed completion in the last iterate is close to $\mc{A}$.
%%%%%%%%%%%%%%%%%%%%%%%%%%%%%
\begin{figure}[htbp]
    \centering
    \newcommand{\imgwidth}{0.18\linewidth}
    \vspace{2mm}
    \begin{tabular}{@{}>{}m{\labelwidth}@{\hspace{3em}} m{\imgwidth} m{\imgwidth} m{\imgwidth} m{\imgwidth}@{}}
    {\tt den} & \quad observation & \quad\ rank $\le  1$ & \quad\ rank $\le  2$ & \quad\ rank $\le  3$ \\
    $0.5$ & \includegraphics[width=\linewidth]{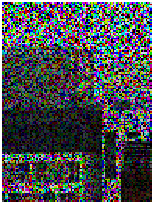} & \includegraphics[width=\linewidth]{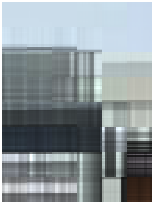} & \includegraphics[width=\linewidth]{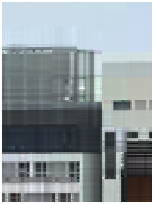} & \includegraphics[width=\linewidth]{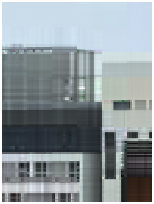} \\
    $0.7$ & \includegraphics[width=\linewidth]{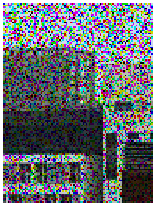} & \includegraphics[width=\linewidth]{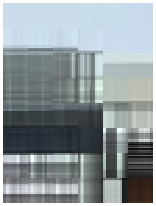} & \includegraphics[width=\linewidth]{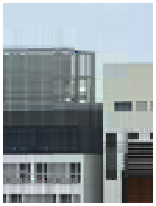} & \includegraphics[width=\linewidth]{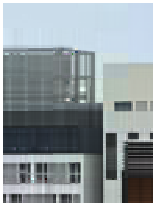} \\
    \end{tabular}
    \caption{Results of completing a real-world image tensor}
    \label{fig:img_real}
\end{figure}

\section{Conclusions}
\label{sc:con}

This paper studies the rank-$1$ tensor completion problem for cubic order tensors
when there are noises for observed tensor entries.
First, we propose a robust optimization model for getting a rank-$1$ tensor completion.
It is a biquadratic optimization problem with sphere constraints.
When the observed tensor is sufficiently close to a rank-$1$ tensor,
we show that the optimizer of this biquadratic optimization
will give a close rank-$1$ completing tensor.
Second, we give an efficient convex relaxation for solving this biquadratic optimization.
When the optimizer matrix $G[y^*]$ is separable, we show
how to get optimizers for the biquadratic optimization.
In particular, if $\rank\, G[y^*]=1$, then the optimizer matrix $G[y^*]$ is always separable.
When $G[y^*]$ is not separable, we apply its spectral decomposition to
obtain approximate optimizers. We refer to the tight relaxation method in
\cite[Chap.~6]{NiePoly} for such cases.
Numerical experiments are presented to show the efficiency of
this biquadratic optimization model and the proposed convex relaxation.

There are interesting applications of rank one tensor completions,
which can be found in Section~\ref{sc:application}.
An interesting direction for future research is to explore the applications in broader ranges of fields.
Moreover, there are some interesting questions for future work:
\begin{itemize}

\item When all minimizers of (\ref{pri-find-ab}) are singular,
how can we find rank-$1$ tensor completions efficiently?

\item How can we efficiently find rank-$1$ tensor completions
when the tensor order $m \ge 4$?

\end{itemize}

\section*{Availability of data and materials}
\noindent This paper does not analyze or generate any dataset.
\section*{Declarations}
\noindent{Competing Interests.\ } The authors have no competing interests to declare that are relevant to the content of this article.

\section*{Acknowledgements}
\noindent Jiawang Nie is partially supported by the NSF grant DMS-2513254
and the AFOSR grant FA9550-25-1-0298.

\end{document}